\documentclass[12pt]{amsart} 
\usepackage{mathrsfs}
\usepackage{amssymb,amsmath,amsthm,color}
\usepackage{graphicx}
\usepackage{hyperref}
\usepackage{url} 
\usepackage{setspace}
\usepackage{mathtools}
\usepackage[inline]{enumitem}  
\usepackage[margin=1in]{geometry}
\usepackage{comment}
\numberwithin{equation}{section}
\newtheorem{theorem}{Theorem}[section]
\newtheorem{Definition}{Definition}[section]
\newtheorem{Remark}{Remark}[section]
\newtheorem{proposition}{Proposition}[section]
\newtheorem{Lemma}{Lemma}[section]

\newcommand{\R}{\mathbb R}
\newcommand{\C}{\mathbb C}
\newcommand{\Z}{\mathbb Z}

\newcommand{\N}{\mathbb N} 
\begin{document}
\baselineskip16pt
\title[Hartree-Fock equations in $M^{p,q}-$spaces]{The Hartree-Fock equations in  modulation  spaces}
\author{Divyang G. Bhimani}
\address{TIFR  Centre for Applicable Mathematics\\
Bangalore\\
India 560065}
\email{divyang@tifrbng.res.in}

\author{Manoussos Grillakis}
\address{Department of Mathematics\\
University of Maryland\\
College Park\\
MD 20742}

\email{mggrlk@math.umd.edu}

\author{Kasso A. Okoudjou}
\address{Department of Mathematics and Norbert Wiener Center\\
University of Maryland\\
College Park\\
MD 20742}
\email{kasso@math.umd.edu}

\date{\today}

\maketitle
\begin{abstract}  
 We establish both a local and a global well-posedness theories for  the nonlinear Hartree-Fock   equations and its reduced analog in the setting of the modulation spaces on $\R^d$. In addition, we prove similar results when  a harmonic potential is added to the equations. In the process, we prove the boundedeness of certain multilinear operators on products of the modulation spaces which may be of independent interest. 
\end{abstract}

\tableofcontents

\section{Introduction and Description of the Problem}
\subsection{ Motivation}
The Hartree equation, introduced by Hartree in the 1920s, arises as the mean-field limit of large systems of identical bosons, e.g., the Gross-Pitaevskii equation for Bose-Einstein condensates \cite{gross1961structure, pitaevskii1961vortex}, when taking into account the self-interactions of the bosons. A semirelativistic version of the Hartree equation was considered in   \cite{elgart2007mean, lenzmann2007well} for modeling boson stars. The Hartree-Fock equation, also developed by Fock \cite{fock1930naherungsmethode}  describes large systems of identical fermions by taking into account the self-interactions of charged fermions as well as an exchange term resulting from Pauli's principle. A semirelativistic version of the Hartree-Fock equation was developed in \cite{frohlich2007dynamical}  for modeling white dwarfs. The Hartree equation is also used for fermions as an approximation of the Hartree-Fock equation neglecting the impact of their fermionic nature. Hartree and Hartree-Fock equations are used for several applications in many-particle physics \cite[Section 2.2]{lipparini2008modern}.

In   \cite{cabre2014nonlinear, laskin2002fractional}   fractional Laplacians have been applied to model physical phenomena.  It was formulated by Laskin \cite{laskin2002fractional} as a result of extending the Feynman path integral from the Brownian-like to L\'evy-like quantum mechanical paths.  The harmonic oscillator (Hermite operator) $-\Delta+|x|^2$ is a fundamental operator in quantum physics and in  analysis \cite{thangavelu1993lectures}. Hartree-Fock equations with harmonic potential model Bose-Einstein condensates with attractive inter-particle  interactions under a magnetic trap . The isotropic harmonic potential $|x|^2$ describes a magnetic field whose role is to confine the movement of particles.  A class of nonlinear Schr\"odinger equations  with a ``nonlocal'' nonlinearity that we call ``Hartree type"  also occurs in the modeling of quantum semiconductor devices (see \cite{carles2003nonlinear} and the references therein).

 \subsection{Hartree-Fock equations}   Before giving the exact form of the Hatree-Fock equations, we set some notations that will be used through the paper. For two functions $\phi$ and $h$ defined on    $\R$ and $\R^d$ respectively, we set 
 $$\phi (h (D))f=  \mathcal{F}^{-1}e^{it \phi \circ h (\cdot)} \mathcal{F}f$$ where $\mathcal{F}$ denotes the  Fourier
 transform. 

 The Hartree-Fock equation of $N$ particles is given by
\begin{equation}\label{HF}
\begin{cases} i\partial_t \psi_k = \phi (h(D))  \psi_k-\sum_{l=1}^{N} \left(\frac{\kappa}{|x|^{\gamma}} \ast |\psi_l|^2 \right)\psi_{k} + \sum_{l=1}^{N} \psi_{l} \left(\frac{\kappa}{|x|^{\gamma}} \ast \{\overline{\psi_l} \psi_k\} \right),\\
\psi_{k|t=0}=\psi_{0k},
\end{cases}
\end{equation}
where   $t\in \R, \psi_k:\R^d\times \R \to \C,$ $k=1,2,..., N,   0<\gamma<d,$  $ \kappa$ is constant,  and $\ast$ denotes the convolution in $\mathbb R^d.$

 The Hartree factor 
\begin{eqnarray*}
H= \sum_{l=1}^{N} \left(\frac{\kappa}{|x|^{\gamma}} \ast |\psi_l|^2 \right)
\end{eqnarray*}
describes the self-interaction between charged particles as a repulsive force if  $\kappa>0$, and an attractive force   if $\kappa<0.$  The last term on the right side of  \eqref{HF} is the so-called ``exchange term (Fock term)"
\[F(\psi_k)=\sum_{l=1}^{N} \psi_{l} \left(\frac{\kappa}{|x|^{\gamma}} \ast \{\overline{\psi_l} \psi_k\} \right)\]
 which is a consequence of the Pauli principle and thus applies to fermions.  In the mean-field limit $(N\to \infty)$, this term is negligible compared to the Hatree factor.
 In this case,~\eqref{HF} is replaced by the  $N$ coupled equations, the so-called  \textbf{reduced Hartree-Fock equations}:
\begin{equation}\label{RHF}
\begin{cases} i\partial_t \psi_k = \phi (h(D))  \psi_k-\sum_{l=1}^{N} \left(\frac{\kappa}{|x|^{\gamma}} \ast |\psi_l|^2 \right)\psi_{k},\\
\psi_{k|t=0}=\psi_{0k}.
\end{cases}
\end{equation} 
The rigorous time-dependent Hartree-Fock theory has been developed first by Chadam-Glassey \cite{chadam1975global}  for  \eqref{RHF}  with  $\phi(h(D))=-\Delta$ in dimension $d=3.$ In this setting, \eqref{RHF}  is  equivalent   to the von Neumann equation 
\begin{eqnarray}\label{vn}
iK' (t)=\left[G(t), K(t) \right]
\end{eqnarray}
for $K(t)= \sum_{1}^{N} \left| \psi_k(t) \rangle  \langle \psi_k(t) \right|$ and $G(t)=\phi(h(D))+H(x,t),$  see, e.g.,  \cite{komech2014hartree, lewin2014hartree, lewin2015hartree}.
In the above equation, we use Dirac's notation $|u \rangle \langle v|$ for the operator $f\mapsto \langle v,f\rangle u$. The von Neumann  equation~\eqref{vn}  can also  be considered for more general class of density matrices  $K(t)$. For example, one can consider the class of nonnegative self-adjoint trace class operators, for which $K(t)$ satisfies the following conditions:
\begin{eqnarray*}
K^*(t)=K(t), K(t)\leq 1,  \  \text{tr} K =N
\end{eqnarray*}
where the condition  $K(t)\leq 1$  corresponds to the Pauli exclusion principle, and  $N$ is the ``number of particles". 

The well-posedness for  \eqref{vn} was proved by Bove-Da Parto-Fano \cite{bove1974existence, bove1976hartree} for a short-range paire-wise interaction potential $w(x-y)$ instead of  Coulomb  potential   $\frac{1}{|x-y|} $ in  the Hartree factor. The case of Coulomb potential was  resolved by Chadam \cite{chadam1976time}.  Lewin-Sabin \cite{lewin2015hartree} have established the well-posedness for   \eqref{vn} with density matrices of infinite trace for pair-wise interaction potentials  $w\in L^1(\mathbb R^3)$.  However, their investigation did not include the Coulomb potential case.  Moreover, Lewin-Sabin \cite{lewin2014hartree} prove the asymptotic stability for the ground state in dimension $d=2$. 
Recently,  Fr\"ohlich-Lenzmann  \cite{frohlich2007dynamical} and  Carles-Lucha-Moulay \cite{carles2015higher}  studied the local and global well-posedness  for  \eqref{HF} and  \eqref{RHF} in $L^2-$based Sobolev spaces, when $d=3$. The  existence of a global solution to \eqref{HF}  was established in \cite[Theorem 2.2]{frohlich2007dynamical} assuming  sufficiently small initial data.   These results naturally  raise two questions.  First,  could similar results be in other functions spaces?  Second,  is it possible to  obtain  the   existence of global solutions to   \eqref{HF} and \eqref{RHF}  with  any  initial data.

We investigate these two questions in the setting of the modulation spaces $M^{p,q}(\mathbb R^d)$ (to be defined below),  which have recently been considered as spaces of  Cauchy data for certain   nonlinear dispersive equations, see \cite{baoxiang2006isometri, benyi2007unimodular, benyi2009local, Bhimani2018global, bhimani2016cauchy, ruzhansky2016global,  wang2009global, wang2007global, wang2011harmonic}.  Generally  modulation spaces are  considered as low regularity spaces  because  they contain  rougher functions  than functions in any given  fractional Bessel potential space  (see  Proposition  \ref{msk} below). We refer to excellent survey \cite{ruzhansky2012modulation}  and the reference therein for details.

 Taking these considerations into account, we  initiate the study of  \eqref{HF}
and \eqref{RHF}  in modulation spaces. In particular, the two our main results can be stated as follows. 

\begin{theorem}[Local well-posedness]\label{gidi}  Let $N, d\in \N $, and $\gamma \in (0, d)$ be given. Let $X$ be given by 
$$
X=\begin{cases}
M^{p,q}(\mathbb R^d)\quad  \textrm{if}\quad  1\leq p \leq 2, 1\leq q \leq \frac{2d}{d+\gamma}\\
M_s^{p,1}(\mathbb R^d)\quad  \textrm{if}\quad  1<p< \infty, \frac{1}{p}+\frac{\gamma}{d}-1= \frac{1}{p+\epsilon}
\end{cases}
$$ for some $\epsilon>0$  and $s\geq 0$.  Let  $\phi: \R^d\to \R$ be such  that  there exist $m_1, m_2>0$ with 
$$\begin{cases}
 \left| \phi^{(\mu)} (r)\right| \lesssim r^{m_1-\mu} \quad \textrm{if}\quad  r\geq 1\\
  \left| \phi^{(\mu)} (r)\right| \lesssim r^{m_2-\mu}  \quad  \textrm{if}\quad  0<r<1\end{cases}$$ 
for all  $\mu \in \mathbb N_0.$  Furthermore, assume that $h\in C^{\infty}(\mathbb R^d \setminus \{ 0\})$ is either \newline
\noindent (a) a positive function of homogeneous type  of degree $\lambda>0$ with $m_1\lambda \leq 2$, or  \newline
\noindent  (b) $\phi\circ h (\xi):= P(\xi)=\sum_{|\beta|\leq m} c_{\beta} \xi^{\beta}$ is a polynomial with order $m\leq 2.$

Given initial data  $\left(\psi_{01},..., \psi_{0N} \right) \in X^N$, the following statements hold. 
 
  \begin{enumerate}
  \item[(i)] \label{gidi1}There exists $T>0$ depending only on $\|\psi_{01}\|_{X},...,\|\psi_{0N}\|_{X},$ $d$ and $\gamma$ such that \eqref{HF} has a unique local solution
$$(\psi_1,..., \psi_{N})\in \left( C([0, T], X) \right)^N.$$ 
\item[(ii)]  \label{gidi2} There exists $T>0$ depending only on $\|\psi_{01}\|_{X},...,\|\psi_{0N}\|_{X},$ $d$ and $\gamma$ such that \eqref{RHF} has a unique local solution
$$(\psi_1,..., \psi_{N})\in \left( C([0, T], X) \right)^N.$$
\end{enumerate}   
 \end{theorem}
 Our second main result deal with the global well-posedness of these equations. In the statement, we denote by $X_{rad}$,  space of radial functions in the Banach space $X$. 

\begin{theorem}[Global well-posedness]\label{dgt}  Suppose that $\phi$ and $h$ are defined on $\R$ and $\R^d$  respectively such that $\phi\circ h (\xi)= |\xi|^{\alpha}$ for $\xi \in \mathbb R^d$, and where $\alpha>0.$  Assume that  $0 < \gamma < \min\{\alpha, d/2\}$, and that one of the following two statements holds:
\begin{itemize}
\item[(a)]  For $\alpha =2$ and $d\in \mathbb N,$ let 
$$X=\begin{cases}
M^{p,q}(\mathbb R^d)\quad  \textrm{if}\quad  1\leq p \leq 2, 1\leq q \leq \frac{2d}{d+\gamma}\\
M_s^{p,1}(\mathbb R^d)\quad  \textrm{if}\quad  1<p< \infty, \frac{1}{p}+\frac{\gamma}{d}-1= \frac{1}{p+\epsilon}
\end{cases}
$$ for some $\epsilon>0$  and $s\geq 0$.
\item[(b)]  For   $\alpha \in  \left(\frac{2d}{2d-1}, 2\right)$ and $d\geq 2,$  let
$$X=\begin{cases}
M_{rad}^{p,q}(\mathbb R^d)\quad \textrm{if} \quad 1\leq p \leq 2, 1\leq q \leq \frac{2d}{d+\gamma}\\
M_{s}^{p,1}(\mathbb R^d)\cap L_{rad}^2(\mathbb R^d)\quad \textrm{if}\quad 2<p< \infty, \frac{1}{p}+\frac{\gamma}{d}-1= \frac{1}{p+\epsilon}\end{cases}
$$  for some $\epsilon>0$ and $s\geq 0$.
\end{itemize}
Given initial data   $(\psi_{01},...,\psi_{0N})\in X^N,$ the following statements hold. 
\begin{enumerate}
\item[(i)]\label{dgt1}  There exists a  unique global solution of \eqref{HF} such that
$$(\psi_1,..., \psi_{N})\in  \left(C(\mathbb R, X)\cap L^{4\alpha/\gamma}_{loc}(\mathbb R, L^{4d/(2d-\gamma)} (\mathbb R^d)) \right)^N.$$
\item[(ii)]  \label{dg2}  There exists a  unique global solution of \eqref{RHF} such that
$$(\psi_1,..., \psi_{N})\in  \left(C(\mathbb R, X)\cap L^{4\alpha/\gamma}_{loc}(\mathbb R, L^{4d/(2d-\gamma)} (\mathbb R^d)) \right)^N.$$
\end{enumerate}
\end{theorem}

In the case $N=1$,  first author in  \cite[Theorem 1.1]{Bhimani2018global}  established the global well-posedness of~ \eqref{RHF} in  $M^{p,q}(\mathbb R^d) $   when $1\leq p \leq 2,$  and $1\leq q < \frac{2d}{d+\gamma}.$ Part (ii) of Theorem \ref{dgt}  proves this result for  the end point case  for any $N\geq 1.$  We note that  $M^{p,q}(\mathbb R^d) \subset L^p(\mathbb R^d) \ (q\leq \min\{ p, p'\})$  is sharp embedding and up to now we cannot get the global well-posedness of  \eqref{HF} in  $L^{p}(\mathbb R^d) (1\leq p <2)$ but in  $M^{p,q}(\mathbb R^d)$ (Theorem \ref{dgt}).  Noticing for $s>\gamma/2,$ we have  sharp embedding $ H^{s}(\mathbb R^d) \subset M^{2,\frac{2d}{d+\gamma}}(\mathbb R^d)\subset L^2(\mathbb R^d)$ (see Proposition \ref{msk} below), Theorem \ref{dgt} reveals that we can solve \eqref{HF} and \eqref{RHF}  with Cauchy data beyond   in $H^s(\mathbb R^d) (s>\gamma/2).$

\begin{Remark}  The sign of  $\kappa$  in Hartree and Fock terms determines the defocusing and  focusing character of the nonlinearity, but, as we shall see, this character will play no role in our analysis on modulation spaces, as we do not use the  conservation  of energy of \eqref{HF} and \eqref{RHF} to achieve global existence.
\end{Remark}

\subsection{Hartree-Fock equation with harmonic potential}
The  Hartree-Fock equation  with the harmonic  potential  of $N$ particles is given by
\begin{equation}\label{HFHP}
\begin{cases} i\partial_t \psi_k - \left(-\Delta +|x|^2\right) \psi_k= \sum_{l=1}^{N} \left(\frac{\kappa}{|x|^{\gamma}} \ast |\psi_l|^2 \right)\psi_{k} + \sum_{l=1}^{N} \psi_{l} \left(\frac{\kappa}{|x|^{\gamma}} \ast \{\overline{\psi_l} \psi_k\} \right),\\
\psi_{k|t=0}=\psi_{0k}
\end{cases}
\end{equation}
and  the  corresponding reduced Hartree-Fock equation with the  harmonic  potential:
\begin{equation}\label{RHFHP}
\begin{cases} i\partial_t \psi_k - \left(-\Delta +|x|^2\right) \psi_k= \sum_{l=1}^{N} \left(\frac{\kappa}{|x|^{\gamma}} \ast |\psi_l|^2 \right)\psi_{k},\\
\psi_{k|t=0}=\psi_{0k},
\end{cases}
\end{equation}
where   $t\in \R, \psi_k:\R^d\times \R \to \C,$ $k=1,2,..., N,   0<\gamma<d,$  $ \kappa$ is constant.
In this context we  establish the following  result.

\begin{theorem}\label{mtg}  Let
 $ 0<\gamma < \text{min} \{2, d/2\}, d\in \mathbb N$ and $1\leq p \leq \frac{2d}{d+\gamma}.$
Given initial data $ (\psi_{01},...,\psi_{0N})\in \left(M^{p,p}(\mathbb R^{d})\right)^N,$ the following statements hold.   
\begin{enumerate}
\item[(i)]  There exists a  unique global solution of \eqref{HFHP} such that $$ (\psi_1,...,\psi_{N})\in \left( C([0,\infty), M^{p,p}(\mathbb R^{d})) \cap L^{8/\gamma}_{loc}([0,\infty), L^{4d/(2d-\gamma)} (\mathbb R^d)) \right)^N.$$
\item[(ii)] There exists a  unique global solution of \eqref{RHFHP} such that $$ (\psi_1,...,\psi_{N})\in \left( C([0,\infty), M^{p,p}(\mathbb R^{d})) \cap L^{8/\gamma}_{loc}([0,\infty), L^{4d/(2d-\gamma)} (\mathbb R^d)) \right)^N.$$
\end{enumerate}
\end{theorem}

In the case $N=1,$  first author in  \cite[Theorem 1.1]{bhimani2018nonlinear} proved that~\eqref{RHFHP} is globally well-posed in   $M^{p,p}(\mathbb R^d)$ for $1\leq p <\frac{2d}{d+ \gamma}.$   Part (ii) of Theorem \ref{mtg}  establishes this result for  the end point case  for any $N\geq 1.$ 


The rest of the paper is organized as follows.  In Section~\ref{sp}, we introduce  some notations and preliminary results which will be used  in the sequel.  In Section~\ref{te}, we prove the boundedness for Hartree nonlinearity on modulation spaces.   In Section~\ref{pmain} we establish two of our main results, namely Theorems \ref{gidi}, and \ref{dgt}. Finally, in Section~\ref{whfhp} we prove Theorem~\ref{mtg}.

\section{Preliminaries}\label{sp}  
\subsection{Notations} The notation $A \lesssim B $ means $A \leq cB$ for  some constant $c > 0 $, whereas $ A \asymp B $ means $c^{-1}A\leq B\leq cA $ for some $c\geq 1$. Given $a, b\in \R$ we let $a\wedge b = \text{min} \{ a, b\}.$ The symbol $A_{1}\hookrightarrow A_{2}$ denotes the continuous embedding  of the topological linear space $A_{1}$ into $A_{2}.$  Put $\mathbb N_0= \mathbb N \cup \{0\}.$
If $\beta=(\beta_1,\cdots, \beta_d)\in \mathbb N_0^d$ is a multi-index, we set
$$|\beta|=\sum_{1}^d\beta_j,\  \  \beta != \prod_{1}^d \beta_j!, \ \ \  \partial^{\beta}=D^{\beta}=\left(\frac{\partial}{\partial x_1}\right)^{\beta_1}\cdots \left(\frac{\partial}{\partial x_d}\right)^{\beta_d},$$
and if $x= (x_1, \cdots x_d)\in \mathbb R^d,$
$$x^{\beta}= \prod_1^d x_j^{\beta_j}.$$
The   $L^{p}(\mathbb R^{d})$ norm is denoted by 
 $$\|f\|_{L^{p}}=\left( \int_{\R^d} |f(x)|^p dx \right)^{1/p} \ \  (1\leq p < \infty),$$
the $L^{\infty}(\mathbb R^{d})$ norm  is $\|f\|_{L^{\infty}}= \text{ess.sup}_{ x\in \mathbb R^{d}}|f(x)| $. For $1\leq p\leq \infty,$ $p'$ denotes the H\"older conjugate of $p,$ that is, $1/p+1/p'=1.$  
We use  $L_t^{r} (I, X)$ to denote the space-time norm
\[\|u\|_{L^{r}_t(I, X)}=  \left( \int_{I} \|u\|^r_{X}  dt \right)^{1/r},\]
where $I\subset \R$ is an interval and $X$ is a Banach space. The Schwartz space is denoted by  $\mathcal{S}(\mathbb R^{d})$, and, its dual, the space of tempered distributions is  denoted by $\mathcal{S'}(\mathbb R^{d}).$ For $x=(x_1,\cdots, x_d), y=(y_1,\cdots, y_d) \in \mathbb R^d, $ we put $x\cdot y = \sum_{i=1}^{d} x_i y_i.$
Let $\mathcal{F}:\mathcal{S}(\mathbb R^{d})\to \mathcal{S}(\mathbb R^{d})$ be the Fourier transform  defined by  
\begin{eqnarray*}
\mathcal{F}f(w)=\widehat{f}(w)=\int_{\mathbb R^{d}} f(t) e^{- 2\pi i t\cdot w}dt, \  w\in \mathbb R^d.
\end{eqnarray*}
Then $\mathcal{F}$ is an isomorphism on $ \mathcal{S}(\mathbb R^{d})$ which  
uniquely extends to an isomorphism on $ \mathcal{S}'(\mathbb R^d).$

The Fourier-Lebesgue spaces $\mathcal{F}L^p(\mathbb R^d)$ is defined by 
$$\mathcal{F}L^p(\mathbb R^d)= \left\{f\in \mathcal{S}'(\mathbb R^d): \|f\|_{\mathcal{F}L^{p}}:= \|\hat{f}\|_{L^{p}}< \infty \right\}.$$
For $p\in (1, \infty)$ and $s\geq 0$, $W^{s, p}(\R^d)$ will denote standard   Sobolev  space. In particular, 
if $s$ is an  integer, then $W^{s, p}$  consists of  $L^p-$functions with derivatives in  $L^p$ up to order $s$, hence coincides with the $L^p_s-$Sobolev space, also known as Bessel potential
space, defined for  $s\in \R$ by
$$L^p_s(\mathbb R^d)=\left\{f\in \mathcal{S}'(\mathbb R^d): \|f\|_{L^{p}_s}:=\left\| \mathcal{F}^{-1} [\langle \cdot \rangle^s \mathcal{F}(f)]\right\|_{L^p}< \infty \right\},$$
where  $\langle \xi \rangle^{s} = (1+ |\xi|^2)^{s/2} \ (\xi \in \mathbb R^d).$ Note that $L^p_{s_1}(\R^d) \hookrightarrow L^p_{s_2}(\R^d)$ if $s_2\leq s_1.$  

\subsection{Modulation  spaces} Feichtinger  \cite{feichtinger1983modulation} introduced the modulation spaces by imposing integrability conditions on the \emph{short-time Fourier transform} (STFT) of functions or distributions defined on $\mathbb R^d$. To be specific,  the  STFT  of a function $f$ with respect to a window function $g \in {\mathcal S}(\R^d)$ is defined by
\begin{eqnarray*}\label{stft}
V_{g}f(x,w)= \int_{\mathbb R^{d}} f(t) \overline{g(t-x)} e^{-2\pi i w\cdot t}dt,  \  (x, w) \in \mathbb R^{2d}
\end{eqnarray*}
 whenever the integral exists.
For $x, w \in \R^d$ the translation operator $T_x$ and the modulation operator $M_w$ are
defined by $T_{x}f(t)= f(t-x)$ and $M_{w}f(t)= e^{2\pi i w\cdot t} f(t).$ In terms of these
operators the STFT may be expressed as
\begin{eqnarray}
\label{ipform} V_{g}f(x,w)=\langle f, M_{w}T_{x}g\rangle\nonumber
\end{eqnarray}
 where $\langle f, g\rangle$ denotes the inner product for $L^2$ functions,
or the action of the tempered distribution $f$ on the Schwartz class function $g$.  Thus $V: (f,g) \to V_g(f)$ extends to a bilinear form on $\mathcal{S}'(\R^d) \times \mathcal{S}(\R^d)$ and $V_g(f)$ defines a uniformly continuous function on $\R^{d} \times \R^d$ whenever $f \in \mathcal{S}'(\R^d) $ and $g \in  \mathcal{S}(\R^d)$.

\begin{Definition}[Modulation spaces]\label{ms} Let $1 \leq p,q \leq \infty, s \in \R$ and $0\neq g \in{\mathcal S}(\R^d)$. The  weighted  modulation space   $M_s^{p,q}(\R^d)$
is defined to be the space of all tempered distributions $f$ for which the following  norm is finite:
$$ \|f\|_{M_s^{p,q}}=  \left(\int_{\R^d}\left(\int_{\R^d} |V_{g}f(x,w)|^{p} dx\right)^{q/p} (1+|w|^2)^{sq/2} \, dw\right)^{1/q},$$ for $ 1 \leq p,q <\infty$. If $p$ or $q$ is infinite, $\|f\|_{M_s^{p,q}}$ is defined by replacing the corresponding integral by the essential supremum.  For $s=0,$ we write $M^{p,q}_0(\R^d)= M^{p,q}(\R^d).$
\end{Definition}

It is standard to show that this definition is  independent of the choice of 
the particular window function, e.g., see,   \cite[Proposition 11.3.2(c)]{grochenig2013foundations}.

Using a uniform partition of the frequency domain,  
one can obtain an equivalent definition of the modulation spaces  \cite{wang2007global}  as follows. Let  $Q_k$ be the unit cube with the center at  $k$, so $\{ Q_{k}\}_{k \in \mathbb Z^d}$ constitutes a decomposition of  $\mathbb R^d,$ that is, $\mathbb R^d = \cup_{k\in \mathbb Z^{d}} Q_{k}.$
Let   $\rho \in \mathcal{S}(\mathbb R^d),$  $\rho: \mathbb R^d \to [0,1]$  be  a smooth function satisfying   $\rho(\xi)= 1 \  \text{if} \ \ |\xi|_{\infty}\leq \frac{1}{2} $ and $\rho(\xi)=
0 \  \text{if} \ \ |\xi|_{\infty}\geq  1,$ where $|\xi|_{\infty}=\max_{k=1, \hdots, d}|\xi_k|$.  Let  $\rho_k$ be a translate of $\rho,$ that is,
\[ \rho_k(\xi)= \rho(\xi -k) \ (k \in \mathbb Z^d).\]
For each $0\neq k\in \Z$ let 
$$\sigma_{k}(\xi)= \frac{\rho_{k}(\xi)}{\sum_{l\in\mathbb Z^{d}}\rho_{l}(\xi)} $$
and when $k=0$, we simply write $\sigma_0=\sigma.$   Then   $\{ \sigma_k(\xi)\}_{k\in \mathbb Z^d}$ satisfies the following properties
\begin{eqnarray*}
\begin{cases} |\sigma_{k}(\xi)|\geq c, \forall \xi \in Q_{k},\\
\text{supp} \ \sigma_{k} \subset \{\xi: |\xi-k|_{\infty}\leq 1 \},\\
\sum_{k\in \mathbb Z^{d}} \sigma_{k}(\xi)\equiv 1, \forall \xi \in \mathbb R^d,\\
|D^{\alpha}\sigma_{k}(\xi)|\leq C_{|\alpha|}, \forall \xi \in \mathbb R^d, \alpha \in (\mathbb N \cup \{0\})^{d}
\end{cases}
\end{eqnarray*} for some positive constant $c$. 

The frequency-uniform decomposition operators can be defined by 
\begin{eqnarray}\label{wn}
\square_k = \mathcal{F}^{-1} \sigma_k \mathcal{F}.
\end{eqnarray}
For $1\leq p, q \leq \infty, s\in \mathbb R,$  it is known \cite{feichtinger1983modulation} that 
\begin{eqnarray}\label{edfw}
\|f\|_{M^{p,q}_s}\asymp  \left( \sum_{k\in \mathbb Z^d} \| \square_k(f)\|^q_{L^p} (1+|k|)^{sq} \right)^{1/q},
\end{eqnarray}
with natural modifications for $p, q= \infty.$
As observed in \cite{wang2011harmonic}, the frequency-uniform decomposition operators obey an almost orthogonality  relation: for each $k\in \Z$
 \begin{eqnarray}\label{aor}
 \square_k= \sum_{\|\ell \|_{\infty}\leq 1} \square_{k+\ell}\square_{k}
\end{eqnarray}
where $\|\ell\|_{\infty}= \max \{|\ell_i|:\ell_i \in \mathbb Z, i=1,..., d\}.$ 

We now list some basic properties of the modulation spaces.
\begin{Lemma}  \label{rl} Let $p,q, p_{i}, q_{i}\in [1, \infty]$  $(i=1,2), s, s_1, s_2 \in \R.$ Then
\begin{enumerate}
\item[(1)] \label{ir} $M^{p_{1}, q_{1}}_{s_1}(\mathbb R^{d}) \hookrightarrow M^{p_{2}, q_{2}}_{s_2}(\mathbb R^{d})$ whenever $p_{1}\leq p_{2}$ and $q_{1}\leq q_{2}$ and $s_2\leq s_1.$
\item[(2)] \label{el} $M^{p,q_{1}}(\mathbb R^{d}) \hookrightarrow L^{p}(\mathbb R^{d}) \hookrightarrow M^{p,q_{2}}(\mathbb R^{d})$ holds for $q_{1}\leq \text{min} \{p, p'\}$ and $q_{2}\geq \text{max} \{p, p'\}$ with $\frac{1}{p}+\frac{1}{p'}=1.$
\item[(3)]  \label{rcs} $M^{\min\{p', 2\}, p}(\mathbb R^d) \hookrightarrow \mathcal{F} L^{p}(\mathbb R^d)\hookrightarrow M^{\max \{p',2\},p}(\mathbb R^d),  \frac{1}{p}+\frac{1}{p'}=1.$
\item[(4)]  \label{d} $\mathcal{S}(\mathbb R^{d})$ is dense in  $M^{p,q}(\mathbb R^{d})$ if $p$ and $q<\infty.$
\item[(5)] \label{fi} The Fourier transform $\mathcal{F}:M_s^{p,p}(\mathbb R^{d})\to M_s^{p,p}(\mathbb R^{d})$ is an isomorphism.
\item[(6)] The space  $M_s^{p,q}(\mathbb R^{d})$ is a  Banach space.
\item[(7)]  \label{ic}The space $M_s^{p,q}(\mathbb R^{d})$ is invariant under complex conjugation.
\end{enumerate}

\begin{proof}
For the proof of parts (1), (2), (3), and (4)  see  \cite[Theorem 12.2.2]{grochenig2013foundations}, 
\cite[Proposition 1.7]{toft2004continuity},   \cite[Corollary 1.1]{cunanan2015inclusion}   and \cite[Proposition 11.3.4]{grochenig2013foundations} respectively.  The proof  of statement (5) can be derived from  the fundamental identity of time-frequency analysis: 
$$ V_gf(x, w) = e^{- i 2 \pi  x \cdot w } \, V_{\widehat{g}} \widehat{f}(w, -x),$$
which is easy to obtain. The proof of statement (6) is trivial, indeed, we have $\|f\|_{M^{p,q}}=\|\bar{f}\|_{M^{p,q}}.$
\end{proof}
\end{Lemma}

We can obtain examples of functions in the modulation spaces via embedding relations with certain classical functions spaces. For example the following result can be proved.

\begin{proposition}[Examples]\label{msk} The following statements hold.
\begin{enumerate}
\item[(i)] \label{smi} ( \cite{kobayashi2011inclusion},  \cite[Theorem 3.8]{ruzhansky2012modulation}) Let $1\leq p, q \leq \infty,$ $s_1, s_2 \in \R,$ and 
$$\tau (p,q)= \max \left\{ 0, d\left( \frac{1}{q}- \frac{1}{p}\right), d\left( \frac{1}{q}+ \frac{1}{p}-1\right) \right\}.$$Then 
$L^{p}_{s_1}(\R^d) \subset M^{p,q}_{s_2}(\R^d)$ if and only if one of the following conditions is satisfied:
$$ \begin{cases}
q\geq p>1, s_1\geq s_2 + \tau(p,q), or\\
p>q, s_1>s_2+ \tau(p,q), or\\
p=1, q=\infty, s_1\geq s_2 + \tau(1, \infty), or\\
p=1, q\neq \infty, s_1>s_2+\tau (1, q). 
\end{cases}
$$

\item[(ii)]  \label{rsmi} (\cite{kobayashi2011inclusion}, \cite[Theorem 3.8]{ruzhansky2012modulation}) Let $1\leq p, q \leq \infty,$ $s_1, s_2 \in \R,$ and 
$$\sigma (p,q)= \max \left\{ 0, d\left( \frac{1}{p}- \frac{1}{q}\right), d\left( 1-\frac{1}{q}- \frac{1}{p}\right) \right\}.$$Then 
$ M^{p,q}_{s_1}(\R^d) \subset  L^{p}_{s_2}(\R^d)$ if and only if one of the following conditions is satisfied:
$$ \begin{cases}
q\leq p<\infty, s_1\geq s_2 + \sigma(p,q), or\\
p<q, s_1>s_2+ \sigma(p,q), or \\
p=\infty, q=1, s_1\geq s_2 + \sigma(\infty, 1), or \\
p=\infty, q\neq 1, s_1>s_2+\sigma (\infty, q).
\end{cases} $$
 
\item[(iii)] \label{lsi}   For  $1\leq p<2,$  $M^{p,p}(\mathbb R^d) \subset L^p(\mathbb R^d)$ and  there exists $f\in L^{p}(\mathbb R^d)$ such that   $f \notin M^{p,p}(\mathbb R^d).$ 
\item[(iv)] \label{csm} For $s>\frac{\gamma}{2}>0,$ $H^s(\mathbb R^d) \subset M^{2,\frac{2d}{d+\gamma}}(\mathbb R^d)$  and there exists  $f\in M^{2,\frac{2d}{d+\gamma}}(\mathbb R^d)$ such that $f\notin H^s (\mathbb R^d).$ 

\item[(v)]  \label{esw} For $2<p<\infty$ and $s>d \left(1- \frac{1}{p}\right),$ $L^p_s(\mathbb R^d)\subset M^{p,1}(\mathbb R^d)$ and  there exists $f\in M^{p,1}(\mathbb R^d)$ such that $f\notin L^p_s(\mathbb R^d).$
\end{enumerate}
\end{proposition}

\begin{proof} We only give proofs of the last three parts. 
\begin{enumerate}

 \item[(iii)] For $1\leq p<2,$  by part (2) of Lemma \ref{rl}, we have  $M^{p,p}(\mathbb R^d) \subset L^p(\mathbb R^d).$ We claim that  $M^{p,p}(\mathbb R^d)\subsetneq L^p(\mathbb R^d).$ If possible, suppose that claim is not true, that is,  for all $f\in L^p(\mathbb R^d),$ we have  $f\in M^{p,p}(\mathbb R^d).$ It follows that  $L^p(\mathbb R^d)= M^{p,p}(\mathbb R^d).$ But  then by part (5) of  Lemma \ref{rl},  it follows that $L^p(\mathbb R^d)$ invariant under the Fourier transform, which is a contradiction. Hence, the claim.   Similarly, for  $2<p\leq \infty,$ we have   $M^{p,p}(\mathbb R^d)\subsetneq L^p(\mathbb R^d).$
 \item[(iv)]  Noticing  $\tau  \left(2, \frac{2d}{d+\gamma} \right)=\frac{\gamma}{2},$ by part (i), we have   $H^s(\mathbb R^d)\subset M^{2,1}(\mathbb R^d)$ for $s>\gamma/2.$  We claim that  $H^s(\mathbb R^d)\subsetneq M^{2,\frac{2d}{d+\gamma}}(\mathbb R^d)$. If possible, suppose that claim is not true. Then we have $H^s(\mathbb R^d)=M^{2,\frac{2d}{d+\gamma}}(\mathbb R^d)$. But then, noticing   $\sigma \left(2, \frac{2d}{d+\gamma} \right)=
 -\frac{\gamma}{2},$   part (ii) gives contradiction. Hence, the claim.
 \item[(v)] Noticing $\tau(p,1)=d \left(1- \frac{1}{p}\right)$ and $\sigma (p,1)=-\frac{d}{p},$  parts (i) and (ii) give  $L^p_s(\mathbb R^d)\subsetneq  M^{p,1}(\mathbb R^d).$
 \end{enumerate}
\end{proof}

\begin{proposition}(Algebra property, \cite[Theorem 2.4]{toft2004continuity}) \label{gap} Let $s\geq 0$, and  $p,q, p_{i}, q_{i}\in [1, \infty]$, where   $i=0,1,2$. 
If   $\frac{1}{p_1}+ \frac{1}{p_2}= \frac{1}{p_0}$ and $\frac{1}{q_1}+\frac{1}{q_2}=1+\frac{1}{q_0}, $ then
\begin{eqnarray*}\label{prm}
M_s^{p_1, q_1}(\mathbb R^{d}) \cdot M_s^{p_{2}, q_{2}}(\mathbb R^{d}) \hookrightarrow M_s^{p_0, q_0}(\mathbb R^{d})
\end{eqnarray*}
with norm inequality $\|f g\|_{M_s^{p_0, q_0}}\lesssim \|f\|_{M_s^{p_1, q_1}} \|g\|_{M_s^{p_2,q_2}}.$
In particular, the  space $M^{p,q}(\mathbb R^{d})$ is a poinwise $\mathcal{F}L^{1}(\mathbb R^{d})$-module, that is, it satisfies
\begin{eqnarray*}
\|fg\|_{M^{p,q}} \lesssim \|f\|_{\mathcal{F}L^{1}} \|g\|_{M^{p,q}}.
\end{eqnarray*} 
\end{proposition}

\subsection{Modulation space estimates for unimodular Fourier multipliers}
In this section, we consider the boundedness properties of a class of  unimodular Fourier multipliers defined by  
$$ U(t) f(x)=e^{it\phi (h(D))}f(x)=  \int_{\mathbb R^d}  e^{i \pi t \phi \circ h(\xi)}\, \hat{f}(\xi) \, e^{2\pi i \xi \cdot x} \, d\xi$$ for $f\in \mathcal{S}(\mathbb R^d)$, 
where  $\phi\circ h:\mathbb R^d \to \mathbb R$ is the composition function of $h:\mathbb R^d\to \mathbb R$ and $\phi:\mathbb R \to \mathbb R.$ 

\begin{proposition}\label{des}   Let $s\in \mathbb R$ and $1\leq p, q \leq \infty.$ 
\begin{enumerate}
\item[(i)] (\cite[Theorem 1.1]{deng2013estimate})  Assume that   there exist $m_1, m_2>0$ such that  $\phi$ satisfies

$$\begin{cases}
 \left| \phi^{(\mu)} (r)\right| \lesssim r^{m_1-\mu} \quad \textrm{if}\quad  r\geq 1\\
  \left| \phi^{(\mu)} (r)\right| \lesssim r^{m_2-\mu}  \quad  \textrm{if}\quad  0<r<1\end{cases}$$ 
for all  $\mu \in \mathbb N_0$  and $h\in C^{\infty}(\mathbb R^d \setminus \{ 0\})$ is positive homogeneous function  with degree $\lambda>0$. Then we have 
\[\left\| e^{it\phi (h (D))} f \right\|_{M^{p,q}_s} \lesssim  \|f\|_{M^{p,q}_s}  + |t|^{d \left| \frac{1}{2}- \frac{1}{p} \right|} \|f\|_{M^{p,q}_{s+ \gamma (m_1, \lambda)}}\]
where $\gamma (m_1, \lambda)=d (m_1\lambda -2) |1/2-1/p|$.

\item[(ii)] (\cite[Theorems 1 and 2]{chen2012asymptotic}) Let   $h(\xi)=|\xi|$ and $\phi (r)=r^{\alpha}$, with   $1/2 < \alpha \leq 2.$ Then 
\[ \|U(t)f \|_{M^{p,q}}\leq  (1+|t|)^{d\left| \frac{1}{p}-\frac{1}{2} \right|} \|f\|_{M^{p,q}} \]
 \item[(iii)]  (\cite[Proposition 4.1]{wang2007global}) Let  $ 2 \leq p \leq \infty, 1\leq q \leq \infty,$ $h(\xi)=|\xi|$ and $\phi (r)=r^{\alpha} \ (\alpha \geq 2).$  Then $$ \|U(t)f \|_{M^{p,q}}\leq  (1+|t|)^{- \frac{2d}{\alpha}\left( \frac{1}{2}-\frac{1}{p} \right)} \|f\|_{M^{p',q}}$$ 
\end{enumerate}
\end{proposition}

Another important class of unimodular Fourier multipliers that is not covered by Proposition~\ref{des}, are  the  so-called Fourier multiplier with polynomial symbol.
Specifically, for $f\in \mathcal{S}(\mathbb R^d)$ let
$$ U(t) f(x)=e^{itP(D)}f(x)=  \int_{\mathbb R^d}  e^{i \pi t  P(\xi)}\, \hat{f}(\xi) \, e^{2\pi i \xi \cdot x} \, d\xi,$$ 
where  $P(\xi)=\sum_{|\beta|\leq m} c_{\beta} \xi ^{\beta}$ is a polynomial with order $m\geq 1.$ In this setting the following result was proved in \cite{deng2013estimate}.

\begin{proposition} (\cite[Theorems 4.3]{deng2013estimate}) \label{drw} Let  $s\in \mathbb R,$  $1\leq p, q \leq \infty$  and $m\geq 2.$ Then
\begin{eqnarray*}
\|e^{itP (D)} f\|_{M^{p,q}} \lesssim  \|f\|_{M^{p,q}_s} +   |t|^{d \left| \frac{1}{2}- \frac{1}{p} \right|} \|f\|_{M^{p,q}_{s+ \gamma (m)}}
\end{eqnarray*} where 
$\gamma (m)=d (m-2) |1/2-1/p|.$
\end{proposition}

To make the paper self content, we outline the proof of Proposition~\ref{drw} in the particular case  when $P(\xi, \eta)= |\xi|^2- |\eta|^2,$ and note that the general case can be proved similarly. But first, we state a result that provides a criteria for the Fourier multiplier to be bounded on modulation spaces. In particular, it provides an application of the uniform decomposition operators given in~\eqref{wn}.

\begin{proposition} \label{bi} Let $\square_k$ be defined as in \eqref{wn} and $t\in \mathbb R.$ Suppose that there is an integer $M>0$ such that
\begin{eqnarray*}\label{sk}
\|\square_ke^{itP(D)}f\|_{L^1}\lesssim \begin{cases}  |t|^{b_1} \|f\|_{L^1} \quad \textrm{if}\quad |k|<M\\
  |t|^{b_2} \|f\|_{L^1} \quad \textrm{if}\quad   |k|\geq M\end{cases}
\end{eqnarray*}
where $b_1\geq b_2\geq 0$ for all $f\in L^1(\R^d)$. Then we have 
$$\|e^{it P(D)} f\|_{M^{p,q}_s} \lesssim \left(|t|^{2b_1 \left|\frac{1}{p}- \frac{1}{2} \right| } + |t|^{2b_2 \left|\frac{1}{p}- \frac{1}{2} \right| }  \right) \|f\|_{M^{p,q}_{s}}$$ whenever $f\in M^{p,q}_s(\R^d)$. 
\end{proposition}

\begin{proof}  By \eqref{wn} and 
Plancherel theorem, we obtain
\begin{equation*}\label{ap}
\|\square_ke^{itP(D)}f\|_{L^2}=\|\sigma_ke^{itP}\hat{f}\|_{L^2} \lesssim \|f\|_{L^{2}}
\end{equation*}
for all $k\in \mathbb Z^{d}.$  By the  Riesz-Thorin interpolation theorem, and for any   $1\leq p\leq 2,$  we have 

$$ \|\square_ke^{itD}f\|_{L^{p}}\lesssim 
\begin{cases}
|t|^{2b_1\left| \frac{1}{p}-\frac{1}{2}\right|}\|f\|_{L^{p}} \,\,  {\textrm if}\,\,   |k|< M\\
 |t|^{2b_2 \left| \frac{1}{p}-\frac{1}{2}\right|} \|f\|_{L^{p}}\,\,  {\textrm if}\, \,  |k|\geq M
\end{cases}$$
 Using a duality argument, we obtain the above two inequality for all $1\leq p \leq \infty.$  Using \eqref{aor}, for $f\in \mathcal{S}(\mathbb R^{d})$ and $|k|<M,$ we obtain
\begin{eqnarray*}
\|\square_ke^{itP(D)}f\|_{L^{p}}  & \leq & \sum_{\|\ell\|_{\infty}\leq 1} \|\pi_{k+\ell} e^{itP(D)}\square_kf\|_{L^{p}} \lesssim   |t|^{2b_1 \left| \frac{1}{p}-\frac{1}{2}\right|} \|\square_kf\|_{L^{p}}. 
\end{eqnarray*}
Similarly, for $f\in \mathcal{S}(\mathbb R^{d})$ and $|k|\geq M,$ we obtain
\begin{eqnarray*}
\|\square_ke^{itP(D)}f\|_{L^{p}} 
& \lesssim  & |t|^{2b_2 \left| \frac{1}{p}-\frac{1}{2}\right|} \|\square_k f\|_{L^{p}}. 
\end{eqnarray*}
In view of \eqref{edfw}, and the above two inequalities,  we obtain
\begin{eqnarray*}
\|e^{itP(D)}f\|_{M^{p,q}_s} 
& = & \left\{ \sum_{|k|\leq M+1} (1+|k|)^{sq}\| \square_k(e^{itP(D)}f)\|^q_{L^p}\right\}^{1/q}\\
&&+ \left\{ \sum_{|k|> M+1} (1+|k|)^{sq}\| \square_k(e^{itP(D)}f)\|^q_{L^p}\right\}^{1/q}\\
& \lesssim &  |t|^{2b_1 \left| \frac{1}{p}-\frac{1}{2}\right|}\left\{ \sum_{|k|\leq M+1} (1+|k|)^{sq}\| \square_kf\|^q_{L^p}\right\}^{1/q}\\
&& +|t|^{2b_2 \left| \frac{1}{p}-\frac{1}{2}\right|}\left\{ \sum_{|k|> M+1} (1+|k|)^{sq}\| \square_kf\|^q_{L^p}\right\}^{1/q}\\
& \lesssim &  \left( |t|^{2b_1 \left| \frac{1}{p}-\frac{1}{2}\right|} +|t|^{2b_2 \left| \frac{1}{p}-\frac{1}{2}\right|}\right) \|f\|_{M^{p,q}_s}. 
\end{eqnarray*}
This completes the proof.\end{proof}
Now to apply Proposition \ref{bi}, we must have control on the $L^1-$norm of the projection operator
$\|\square_k (e^{itP(D)}f)\|_{L^1(\mathbb R^{d})}. $ 
Since $ \square_k (e^{it P(D)}f)=\mathcal{F}^{-1}(\sigma_k e^{itP})\ast f, $ in view of Young's inequality, it suffices to control the norm $ \| \mathcal{F}^{-1} (\sigma_k e^{itP})\|_{L^{1}}$, which we shall do in next two lemmas.

\begin{Lemma}(\cite[Lemmas 4.1 and 4.2]{deng2013estimate})  \label{fl}  
Let $t\in \mathbb R, P(z)=P(\xi, \eta)=|\xi|^2-|\eta|^2, z=(\xi, \eta)\in \mathbb R^{2d}, k \in \mathbb Z^{2d},$ and $M>0.$   Then  we have 
  $$\|\mathcal{F}^{-1} (\sigma_{k} e^{it P(D)})\|_{L^1(\mathbb R^{2d})} \lesssim \max \{ |t|^d, 1 \}.$$
\end{Lemma}

\begin{proof} Assume that $|k|> M+1.$
We introduce an auxiliary function defined by
$$\Lambda_k(z)= P (z+k)-P(k)-\nabla P(k)\cdot z$$    for all $ k \in \mathbb Z^{2d}.$ Since $L^1-$norm is invariant under translation and modulation, we have 
\begin{eqnarray*}
\|\mathcal{F}^{-1} (\sigma_k(z)e^{itP(z)})\|_{L^{1}(\mathbb R^{2d})}  & = & \|\mathcal{F}^{-1} (\sigma(z)e^{itP(z+k)})\|_{L^{1}(\mathbb R^{2d})}\\
& = &  \|\mathcal{F}^{-1} (\sigma(z)e^{it (\Lambda_k(z) +P(k)+\nabla P(k)\cdot z)})\|_{L^{1}(\mathbb R^{2d})}\\
& = & \|g^{\vee}(x+\nabla P (k))\|_{L_x^1(\mathbb R^{2d})}\\
& = &  \|g^{\vee}\|_{L^{1}(\mathbb R^{2d})},
\end{eqnarray*}
where $g(z)=\sigma (z)e^{it\Lambda_k(z)}.$
Thus, to prove Lemma \ref{fl},  it suffices to prove 
$$ \|g^{\vee}\|_{L^1} \lesssim \max \{ t^{d}, 1 \}$$
for $t>0.$
We consider
\begin{eqnarray}\label{ol}
\|g^{\vee} \|_{L^1(\mathbb R^{2d})} & = & \int_{|x|\leq t} \left| \int_{\mathbb R^{2d} }\sigma (z) e^{it \Lambda_k(z)} e^{ixz} dz \right|dx  + \int_{|x|>t} \left| \int_{\mathbb R^{2d} }\sigma (z) e^{it \Lambda_k(z)} e^{ixz} dz \right|dx \nonumber \\
& : = & I_1 + I_2.
\end{eqnarray}
By Cauchy-Schwarz inequality and  Plancherel's Theorem, we have 
\begin{eqnarray}\label{p1}
I_1 & \lesssim & \left( \int_{|x|\leq t} 1 dx\right)^{1/2} \left(\int_{|x|\leq t} \left| \int_{\mathbb R^{2d} }\sigma (z) e^{it \Lambda_k(z)} e^{ixz} dz \right|^{2}dx\right)^{1/2}\nonumber \\
& \lesssim & t^{d}\|g^{\vee}\|_{L^2(\mathbb R^{2d})}\lesssim t^{d}.
\end{eqnarray}
Now we concentrate on $I_2.$  For $j \in \{1,2, \cdots, 2d\},$ let 
$$E_t=\{x\in \mathbb R^{2d}: |x|>t\},$$
$$E_{j,t}=\{x\in E_t: |x_j|\geq |x_l| \ \text{for all} \    l\neq j \}.$$
We note that
\begin{eqnarray*}
I_2  & \lesssim  &  \sum_{j=1}^{2d} \int_{E_j,t}  \left| \int_{\mathbb R^{2d} }\sigma (z) e^{it \Lambda_k(z)} e^{ixz} dz \right| dx :=\sum_{j=1}^{2d}I_{2j} .
\end{eqnarray*}
Since  $\sigma$ is compactly supported and $\Lambda_k$ is a smooth function, performing integration by parts and using Plancherel's theorem, we obtain that  for each $j\in \{1,2,..., 2d\}$ 
\begin{eqnarray*}
I_{2j}  & \lesssim & \int_{E_{j,t}}  \frac{1}{|x_j|^L}  \left| \int_{\mathbb R^{2d} } e^{ixz} D_{j}^{L}(\sigma(z) e^{it \Lambda_k(z)})dz \right| dx\\
& \lesssim &   \left(\int_{E_{j,t}}  \frac{1}{|x_j|^{2L}} dx\right)^{1/2}  \|D_j^{L}(\sigma e^{it\Lambda_k(z)})\|_{L^2(\mathbb R^{2d})}\\
& \lesssim &  \left( \int_{|x|>t} \frac{1}{|x|^{2L}} dx \right)^{1/2} \|D_j^{L}(\sigma e^{it\Lambda_k(z)})\|_{L^2(\mathbb R^{2d})}\\
& = & t^{-d} t^{-L} \|D_j^{L}(\sigma e^{it\Lambda_k(z)})\|_{L^2(\mathbb R^{2d})}
\end{eqnarray*}
where we choose $L>d$ as an integer.  Where we have used the fact that since $|x|^2= \sum_{j=1}^{2d}x_j^2\leq 2d|x_j|^2$ for $x\in E_{j,t}, $ we have $|x_j|^{-2L}\lesssim |x|^{-2L}$.  
Consequently,  we have 
\begin{eqnarray}\label{p2}
I_2 \lesssim t^{d} t^{-L} \sum_{j=1}^{2d} \|D_j^{L}(\sigma e^{it\Lambda_k(z)})\|_{L^2(\mathbb R^{2d})}.
\end{eqnarray}
Next, we claim that $$\sum_{j=1}^{2d} \|D_j^{L}(\sigma e^{it\Lambda_k(z)})\|_{L^2(\mathbb R^{2d})}\lesssim t^{L}.$$  Once this  claim is established, the proof of the lemma will follow from  \eqref{ol}, \eqref{p1}, and  \eqref{p2}.

We now give a proof of this  claim. To this end, 
we note that by Taylor's and Leibniz formula, we have 
\begin{eqnarray}\label{at}
\Lambda_k(z)= 2 \sum_{|\beta|=2} \frac{z^{\beta}}{\beta !} \cdot \int_0^1 (1-s)D^{\beta} P(k+sz) ds,
\end{eqnarray}
and
\begin{eqnarray}\label{al}
D^{\gamma} \Lambda_k(z)= \sum_{\gamma_1+\gamma_2=\gamma}  \sum_{|\beta|=2} C_{\beta, \gamma_1, \gamma_2} D^{\gamma_1} z^{\beta}\cdot \int_0^1 (1-s)D^{\beta +\gamma_2} P(k+sz)ds.
\end{eqnarray} 
Since $P(z)$ is a polynomial of order 2, there exists $C_{\gamma}$  such that  
\begin{eqnarray}\label{io}
|D^{\gamma}P(z)| \leq C_{\gamma} |z|^{2-|\gamma|}
\end{eqnarray}
for all $\gamma \in \mathbb N_0^{2d}.$ We note that for $z\in \text{supp}  \ \sigma,$ and $ s\in [0,1],$ we have $|k+sz| \lesssim |k|,$ 
and in view of \eqref{at}-\eqref{io}, we have that for all $|k|>M+1$
$$|D^{\gamma} \Lambda_k(z)| \lesssim C_{\beta, \gamma}\sum_{\gamma_1+\gamma_2=\gamma}  \sum_{|\beta|=2} | D^{\gamma_1} z^{\beta}| \int_0^1 |k|^{-|\gamma_2|}ds \lesssim 1, \  \ \gamma \in \mathbb N_0^{2d},$$ 
which implies that 
\begin{eqnarray}\label{hw}
|D^{\gamma} e^{it \Lambda_k(z)}| & = & \left |  \sum_{l=1}^{|\gamma|} \sum_{|v_l|=|\gamma|} C_{v} t^{l} \Lambda_1^{(v_1)}\cdots \Lambda_k^{(v_l)} \right|\nonumber \\
& \lesssim & \sum_{l=1}^{|\gamma|} t^{l} \lesssim t^{|\gamma|}
\end{eqnarray}
for all $\gamma \in \mathbb N^{2d},$  where for each $l \in \{1, \cdots, |\gamma|\},$ $v_l=(v_1,\cdots, v_l) \in \mathbb N^l.$ 
For fixed $j,$ by Leibniz formula, we have 
\[D_j^{L}(\sigma e^{it\Lambda_k(z)})=\sum_{n=0}^{L}D_j^n (e^{it\Lambda_k(z)}) D_j^{L-n}(\sigma (z)).\]
Using this and \eqref{hw}, we obtain
\begin{eqnarray*}
\sum_{j=1}^{2d} \|D_j^{L}(\sigma e^{it\Lambda_k(z)})\|_{L^2(\mathbb R^{2d})} & \lesssim & \sum_{j=1}^{2d}\sum_{n=0}^{L} t^n \|D_{j}^{L-n}\sigma\|_{L^{2}(\mathbb R^{2d})}\\
 & \lesssim & \sum_{n=0}^L t^n \left( \sum_{j=1}^{2d} \|D_j^{L-n}\sigma \|_{L^2 (\mathbb R^{2d})}\right) \lesssim t^{L}.
\end{eqnarray*}
 This proves the claim when  $|k|>M+1.$ The case $|k|\leq M+1$ can be consider similarly (see e.g. \cite[Lemma 4.2]{deng2013estimate}).
\end{proof}

\begin{proof}[ Sketch Proof of Proposition~\ref{drw}] 
Taking Proposition \ref{bi} and Lemma \ref{fl} into account, the proof follows when $P(\xi, \eta)=|\xi|^2-|\eta|^2, (\xi, \eta)\in \mathbb R^{2d}$. The general case can be done similarly.
\end{proof}

\section{Trilinear $M^{p,q}$ estimates}\label{te}
One of the main technical results needed to prove our main result is establishing a trilinear estimate for the following Hartree type  trilinear operator. For  $0<\gamma <d$, let 
\begin{eqnarray*}
H_{\gamma}(f, g, h):= \left( |\cdot|^{-\gamma} \ast (f \bar{g}) \right) h  
\end{eqnarray*}
where $f, g, h \in \mathcal{S}(\mathbb R^d).$

\begin{proposition} \label{t1}  Let $0<\gamma <d,$  $1\leq p \leq 2,$ and $ 1\leq q \leq \frac{2d}{d+\gamma}$. 
Given   $f,g, h \in M^{p, q} (\mathbb R^d),$   then  $H_{\gamma}(f, g, h) \in M^{p, q} (\R^d)$, and the following estimate  holds $$\|H_{\gamma}(f, g, h)\|_{M^{p, q}} \lesssim \|f\|_{M^{p, q}}  \|g\|_{M^{p, q}}  \|h\|_{M^{p, q}}.$$
\end{proposition}

\begin{proof}  By Proposition \ref{gap}, we have
\begin{eqnarray*}
\| H_{\gamma}(f, g,h)\|_{M^{p,q}} &  \lesssim &  \|  |\cdot|^{-\gamma} \ast (f \bar{g})\|_{M^{\infty, 1}} \|h\|_{M^{p,q}}\\
& \lesssim &   \|  |\cdot|^{-\gamma} \ast (f \bar{g})\|_{\mathcal{F}L^{1}} \|h\|_{M^{p,q}}.
\end{eqnarray*}  We note that
\begin{eqnarray*}
\left| |\xi|^{-(d-\gamma)} \widehat{fg}(\xi) \right| & = & \frac{1}{|\xi|^{d-\gamma}} \left|  \int_{\mathbb R^d} \hat{f}(\xi - \eta) \widehat{\bar{g}}(\eta) d\eta \right|\\
& \leq &  \frac{1}{|\xi|^{d-\gamma}}  \int_{\mathbb R^d} |\hat{f}(\xi - \eta)| | \widehat{\bar{g}}(\eta)| d\eta 
\end{eqnarray*}
and integrating with respect to $\xi,$ we get
\begin{eqnarray*}
  \|  |\cdot|^{-\gamma} \ast (f \bar{g})\|_{\mathcal{F}L^{1}} & \lesssim & \int_{\R^d} \int_{\R^d} \frac{|\hat{f}(\xi_1)| |\widehat{\bar{g}}(\xi_2)|}{|\xi_1- \xi_2|^{d-\gamma}} d\xi_1 d\xi_2= \left\langle | I^{\gamma}\hat{f}|,| \widehat{\bar{g}}|\right \rangle_{L^2(\R^d)} 
\end{eqnarray*}
where $I^{\gamma}$ denotes the Riesz potential of order $\gamma$:
\[
I^{\gamma}\hat{f}(x)=C_{\gamma}\int_{\mathbb R^d}\frac{\hat{f}(y)}{|x-y|^{d-\gamma}} dy.\] By H\"older and  Hardy-Littlewood Sobolev  inequalities and Lemma \ref{rl}, we have 
\begin{eqnarray*}
\|  |\cdot|^{-\gamma} \ast (f \bar{g})\|_{\mathcal{F}L^{1}} &= & \|I^{\gamma}\hat{f}\|_{L^{\frac{2d}{d-\gamma}}} \|\hat{\bar{g}}\|_{L^{\frac{2d}{d+\gamma}}}\\
  & \lesssim  & \|\hat{f}\|_{L^{\frac{2d}{d+\gamma}}} \|\hat{\bar{g}}\|_{L^{\frac{2d}{d+\gamma}}}\\
   & \lesssim  & \|f\|_{M^{\min\left \{ \frac{2d}{d-\gamma}, 2 \right\}, \frac{2d}{d+\gamma}}} \|g\|_{M^{\min \left\{ \frac{2d}{d-\gamma}, 2 \right\},  \frac{2d}{d+\gamma}}}\\
   & = & \|f\|_{M^{2, \frac{2d}{d+\gamma}}} \|g\|_{M^{2, \frac{2d}{d+\gamma}}}\\
   & \lesssim &  \|f\|_{M^{p, q}} \|g\|_{M^{p, q}}.
\end{eqnarray*}
This completes the proof. 
\end{proof}
We next prove a related result for weighted modulation spaces $M^{p,q}_s$.

\begin{proposition} \label{fip} Assume that $0<\gamma < d.$ The following statements hold
\begin{enumerate}
\item[(i)] If   $
1< p_1< p_2< \infty $   with
$\frac{1}{p_1}+\frac{\gamma}{d}-1= \frac{1}{p_2}$ and 
$1\leq  q \leq \infty, s\geq 0.$  For any $f\in M^{p_1,q}_s(\mathbb R^d),$ we have 
$\| |\cdot|^{-\gamma} \ast f\|_{M^{p_2,q}_s} \lesssim \|f\|_{M^{p_1,q}_s}.$
\item[(ii)]  Let $1< p<\infty $ and $\frac{1}{p}+\frac{\gamma}{d}-1= \frac{1}{p+\epsilon}$
 for some $\epsilon>0.$ For any $f,g,h \in M^{p,1}_s(\mathbb R^d),$ we have  $$\|H_{\gamma}(f, g, h)\|_{M^{p,1}_s} \lesssim \|f\|_{M_s^{p,1}}  \|g\|_{M_s^{p,1}}  \|h\|_{M_s^{p,1}}.$$ 
\end{enumerate}
\end{proposition}

\begin{proof} We may rewrite the STFT as 
$V_{g}(x,w)= e^{-2\pi i  x \cdot w} (f\ast M_wg^*)(x)$  where $g^*(y) = \overline{g(-y)}.$ 
\begin{enumerate}
\item[(i)] Using Hardy-Littlewood-Sobolev inequality, we obtain
\begin{eqnarray*}
\||\cdot|^{-\gamma}\ast f\|_{M^{p_2,q}_s} & =  &  \left \|   \left \|  |\cdot|^{-\gamma} \ast  (f \ast M_{w}g^{*}) \right \|_{L^{p_2}}  \langle w \rangle^s  \right\|_{L^{q}_w}\\
& \lesssim &   \left  \| \|   f \ast M_{w}g^{*}\|_{L^{p_1}}  \langle w \rangle^s  \right\|_{L^{q}_w}\\
& \lesssim & \|f\|_{M^{p_1,q}_s}.
\end{eqnarray*}
This completes the proof of part (i).
\item[(ii)]   By Proposition  \ref{gap} and part (1) of  Lemma \ref{rl}, we have 
\begin{eqnarray*}
\| H_{\gamma}(f, g, h)\|_{M_s^{p,1}} & \lesssim & \| |\cdot|^{-\gamma}\ast (fg)\|_{M_s^{\infty,1}} \|h\|_{M_s^{p,1}}\\
& \lesssim & \| |\cdot|^{-\gamma} \ast (f g)\|_{M_s^{p+\epsilon,1}} \|h\|_{M_s^{p,1}},
\end{eqnarray*}
for some $\epsilon>0.$
By  part (i) and Proposition \ref{gap}, we have  $ \|T_{\gamma}(fg)\|_{M_s^{p+\epsilon,1}} \lesssim \| fg\|_{M_s^{p,1}} \lesssim \|f\|_{M_s^{p,1}}\|g\|_{M_s^{p,1}} .$  
\end{enumerate}
\end{proof}
The following result immediately follows.

\begin{proposition}\label{ds1} Let $1< p<\infty $ and $\frac{1}{p}+\frac{\gamma}{d}-1= \frac{1}{p+\epsilon}$
 for some $\epsilon>0.$  For any $f, g,h \in M^{p,1}(\mathbb R^d) \cap L^2(\mathbb R^d),$ we have  $$ \|(|\cdot|^{-\gamma}\ast (f\bar{g}))h\|_{M^{p,1}\cap L^2} \lesssim \|f\|_{M^{p,1} \cap L^2}  \|g\|_{M^{p,1}\cap L^2}  \|h\|_{M^{p,1}\cap L^2}.$$
\end{proposition}

\begin{proof}  By part (2) of Lemma \ref{rl}, we have 
\begin{eqnarray*}
\|( |\cdot|^{-\gamma} \ast f \bar{g})h\|_{M^{p,1}\cap L^2} & := & \| (|\cdot|^{-\gamma}\ast (f\bar{g}))h\|_{M^{p,1}} + \| (|\cdot|^{-\gamma}\ast (f\bar{g}))h\|_{L^2} \\
& \lesssim & \|f\|_{M^{p,1}}  \|g\|_{M^{p,1}}  \|h\|_{M^{p,1}} + \| |\cdot|^{-\gamma} \ast (f\bar{g})\|_{L^{\infty}} \|h\|_{L^2}\\
& \lesssim &  \|f\|_{M^{p,1}}  \|g\|_{M^{p,1}}  \|h\|_{M^{p,1}} + \| |\cdot|^{-\gamma} \ast (f\bar{g})\|_{M^{\infty, 1}} \|h\|_{L^2}\\
& \lesssim &  \|f\|_{M^{p,1}}  \|g\|_{M^{p,1}}  \|h\|_{M^{p,1}} + \| |\cdot|^{-\gamma} \ast (f\bar{g})\|_{M^{p+\epsilon, 1}} \|h\|_{L^2}\\
& \lesssim &  \|f\|_{M^{p,1}}  \|g\|_{M^{p,1}}  \|h\|_{M^{p,1}} +  \|f\|_{M^{p,1}} \|g\|_{M^{p,1}}  \|h\|_{L^2}\\
& \lesssim & \|f\|_{M^{p,1}\cap L^2}  \|g\|_{M^{p,1}\cap L^2}  \|h\|_{M^{p,1}\cap L^2}.
\end{eqnarray*}
This completes the proof.
\end{proof}

We will also need the following result. 

\begin{Lemma}\label{cl} Let  $0<\gamma <d.$
\begin{enumerate}
\item[(i)]  Let $1\leq p \leq 2, 1\leq q \leq  \frac{2d}{d+\gamma}.$ For any $ f,g \in M^{p,q}(\mathbb R^{d})$, we have
$$\| (|\cdot|^{-\gamma}\ast |f|^{2})f - (|\cdot|^{-\gamma}\ast |g|^{2})g\|_{M^{p,q}} \lesssim  (\|f\|_{M^{p,q}}^{2}+\|f\|_{M^{p,q}}\|g\|_{M^{p,q}}+ \|g\|_{M^{p,q}}^{2}) \|f-g\|_{M^{p,q}}.$$
\item[(ii)]   Let $1< p<\infty $ and $\frac{1}{p}+\frac{\gamma}{d}-1= \frac{1}{p+\epsilon}$
 for some $\epsilon>0.$  For any $f, g \in M^{p,1}(\mathbb R^d)\cap L^2(\mathbb R^d),$ we have 
\begin{eqnarray*}
 \|(|\cdot|^{-\gamma}\ast |f|^{2})f - (|\cdot|^{-\gamma}\ast |g|^{2})g\|_{M^{p,1}\cap L^2}  &\lesssim  & (\|f\|_{M^{p,1}\cap L^2}^{2}+\|f\|_{M^{p,1}\cap L^2}\|g\|_{M^{p,1}\cap L^2}\\
 &&+ \|g\|_{M^{p,1}\cap L^2}^{2}) \|f-g\|_{M^{p,1}\cap L^2}.
\end{eqnarray*}
\end{enumerate}
\end{Lemma}

\begin{proof}
Notice that 
\begin{eqnarray*}
 \|( |\cdot|^{-\gamma}\ast |f|^{2})(f-g)\|_{M^{p,1} \cap L^2}  &  \lesssim & \|f\|_{M^{p,1}\cap L^2}^2 \|f-g\|_{M^{p,1}\cap L^2},
\end{eqnarray*}
and 
\begin{eqnarray*}
 \|( |\cdot|^{-\gamma} \ast (|f|^{2}- |g|^{2}))g\|_{M^{p,1}\cap L^2}  & \lesssim & \left( \|f\|_{M_s^{p,1}} \|g\|_{M^{p,1}} + \|g\|_{M^{p,1}}^2\right) \|f-g\|_{M^{p,1}}\\
 && +  \| |\cdot|^{-\gamma} \ast (|f|^{2}- |g|^{2}))g\|_{L^2}\\
 & \lesssim &  \left( \|f\|_{M_s^{p,1}\cap L^2} \|g\|_{M^{p,1}} + \|g\|_{M^{p,1}\cap L^2}^2\right) \|f-g\|_{M^{p,1}\cap L^2}.
 \end{eqnarray*}
 This together with the following identity
$$( |\cdot|^{-\gamma}\ast |f|^{2})f- ( |\cdot|^{-\gamma}\ast |g|^{2})g= ( |\cdot|^{-\gamma}\ast |f|^{2})(f-g) + ( |\cdot|^{-\gamma} \ast (|f|^{2}- |g|^{2}))g, $$
gives  the desired inequality.
\end{proof}

\section{Proofs of main results}\label{pmain}

\subsection{Local well-posedness for Hartree-Fock equations}\label{lhf}
We can now prove our main results, beginning with Theorem~\ref{gidi}.

\begin{proof}[Proof of Theorem \ref{gidi}]
By Duhamel's principle, we  rewrite  the Cauchy problem \eqref{HF} in an integral form: for $k=1,..., N,$
$$\Psi_k(\psi_1,..., \psi_N):=\psi_{k}(t) = U(t)\psi_{0k} -i \int_0^tU(t-s) (H\psi_k)(s) ds + i \int_0^t U(t-s) (F_k(\psi_k))(s) ds.$$


We shall show that $\Psi$ has a unique fixed point in an appropriate function space, for small  $t.$ For this, we consider Banach space  $\left(C([0, T], X)\right)^N,$ with the norm
\[ \|u\|_{ \left( C([0, T], X) \right)^N} = \max_{1\leq k \leq N}  \sup_{t\in [0, T]} \|u_k(t)\|_{X}\]
where $u=(u_1,...,u_N)\in \left( C([0, T], X) \right) ^N.$ By Propositions \ref{des} and \ref{drw}, we have 
\begin{eqnarray*}
\|U(t) \psi_{0k}\|_{X} \lesssim  C_T \|\psi_{0k}\|_{X}
\end{eqnarray*}
where $C_{T}= C (1+|t|)^{d\left| \frac{1}{p}-\frac{1}{2} \right|}.$ By   Minkowski's inequality for integrals, Propositions \ref{des} and \ref{drw} and Propositions~\ref{t1}, and~\ref{fip},  we  obtain

\begin{eqnarray*}
\left\|\int_0^tU(t-s) (H\psi_k)(s) ds\right\|_{X} & \lesssim &  T C_{T}
 \sum_{\ell=1}^{N}  \left\| \left( \frac{\kappa}{|x|^{\gamma}} \ast |\psi_{\ell}|^2\right)\psi_{k} \right\|_{X}\\
& \lesssim &  T C_T\sum_{\ell=1}^{N} \|\psi_{k}\|_{X}\|\psi_{\ell}\|_{X}^2.
\end{eqnarray*}
Similarly,

\begin{eqnarray*}
\left\|\int_0^tU(t-s) (F\psi_k)(s) ds\right\|_{X} 
& \lesssim &  T C_T\sum_{\ell=1}^{N} \|\psi_{k}\|_{X}\|\psi_{\ell}\|_{X}^2.
\end{eqnarray*}
Thus, we have
\begin{eqnarray*}
\|\Psi_{k}\|_{L^{\infty}([0, T], X)} & \lesssim   &  C_T \left(  \|\psi_{0k}\|_{X}  + c T \sum_{\ell=1}^{N} \|\psi_{k}\|_{X}\|\psi_{\ell}\|_{X}^2 \right)
\end{eqnarray*}
for some universal constant $c.$ 

For $M>0,$ put
\begin{eqnarray*}
B_{T, M} & = & \left\{ (\psi_{1},..., \psi_{N}) \in  (C([0, T], X)^{N}: \|\psi_{k}\|_{L^{\infty}([0, T], X)} \leq  M \ \   \text{for} \ k=1,..., N  \right\}
\end{eqnarray*}
which is the closed ball of radius $M$ and centered at the origin in $\left(C([0, T], X)\right)^N$. 
Next, we show that the mapping  $\Psi_k$ takes  $B_{T, M}$ into itself for suitable choice of  $M$ and small  $T>0$. Indeed, if we take  $M= 2C_{T} \max \{ \|\psi_{0k}\|_{X}: i=1,...,N\}$ and $\bar{\psi}= (\psi_1,..., \psi_N)\in B_{T,M},$ we obtain
\begin{eqnarray*}
\|\Psi_k(\bar{\psi}) \|_{ C([0,T], X)} \leq \frac{M}{2} + c C_{T} M^3
\end{eqnarray*}
for all $k=1,...,N.$
We choose a $T$ such that  $cC_{T}M^2 \leq 1/2,$ that is, $T\leq \tilde{T}(\|\psi_{01}\|_{X},..., \|\psi_{0N}\|_{X}, d, \gamma)$  and as a consequence we have 
\[ \|\Psi_k(\bar{\psi})\|_{C([0,T], X)} \leq  M  \   \text{for all} \   k=1,...,N. \]
So $B_{T, M}$ is invariant under  the action of $\Psi$ provided that $T>0$ is sufficiently small. Up to diminishing $T,$ contraction follows readily, since $H_{\gamma}$ is a trilinear operator. So there exist  a unique (in $ B_{T, M}$)  fixed point for $\Psi,$ that is, a solution to \eqref{HF}. This completes the proof of Theorem \ref{gidi} part (i). Similarly, we can produce the proof of Theorem \ref{gidi} part (ii) of which we omit the details. 
\end{proof}

\subsection{Global Well-posedness for Hartree-Fock Equations}\label{ghf}
In this section we prove Theorem \ref{dgt}.

\begin{Definition} A pair $(q,r)$ is $\alpha-$fractional admissible if  $q\geq 2, r\geq 2$ and
$$\frac{\alpha}{q} =  d \left( \frac{1}{2} - \frac{1}{r} \right).$$ 
\end{Definition}
We recall the following results. For details, see \cite{keel1998endpoint, guo2014improved}.

\begin{proposition}[Strichartz estimates] \label{fst}   Denote
$$DF(t,x) :=  e^{-it(-\Delta)^{\alpha/2}}\phi(x) +  \int_0^t U(t-\tau )F(\tau,x) d\tau.$$
\begin{enumerate}
\item[(i)]  \label{fst1} Let $\phi \in L^2(\mathbb R^d),$ $d\in \mathbb N$ and $\alpha=2.$   Then for any time slab $I$ and  admissible pairs $(p_i, q_i)$, $i=1,2,$ 
there exists  a constant $C=C(|I|, q_1)$ such that for all intervals $I \ni 0, $
$$ \|D(F)\|_{L^{p_1, q_1}_{t,x}}  \leq  C \|\phi \|_{L^2}+   C  \|F\|_{L^{p_2', q_2'}_{t,x}}, \ \forall F \in L^{p_2'} (I, L^{q_2'})$$  where $p_i'$ and $ q_i'$ are H\"older conjugates of $p_i$ and $q_i$
respectively \cite{keel1998endpoint}.

\item[(ii)] \label{fst2} Let  $d\ge 2$ and  $\frac{2d}{2d-1} < \alpha < 2.$  Assume that $\phi$ and $F$ are radial. Then for any time slab $I$ and  admissible pairs $(p_i, q_i)$, $i=1,2,$ 
there exists  a constant $C=C(|I|, q_1)$ such that for all intervals $I \ni 0, $
$$ \|D(F)\|_{L^{p_1, q_1}_{t,x}}  \leq  C \|\phi \|_{L^2}+   C  \|F\|_{L^{p_2', q_2'}_{t,x}}, \ \forall F \in L^{p_2'} (I, L^{q_2'})$$  where $p_i'$ and $ q_i'$ are H\"older conjugates of $p_i$ and $q_i$ \cite[Corollary 3.4]{guo2014improved}.
\end{enumerate}
\end{proposition}

We first establish the following preliminary results. 

\begin{proposition}\label{miF}  Let $\phi\circ h (\xi)= |\xi|^{\alpha}$ where  $\xi \in \mathbb R^d, \alpha>0$  and  $0 < \gamma < \min\{\alpha, d\}.$
\begin{enumerate}
\item[(i)] \label{wrc} Let  $d\in \mathbb N$ and  $\alpha =2.$  If $ (\psi_{01}, ...., \psi_{0N}) \in \left(L^2(\mathbb R^d)\right)^N$ then \eqref{HF} has a unique global solution 
$$ (\psi_1,...,\psi_N)   \in  \left(C(\mathbb R, L^{2}(\mathbb R^d))\cap L^{4\alpha/\gamma}_{loc}(\mathbb R, L^{4d/(2d-\gamma)} (\mathbb R^d))\right)^N.$$ 
In addition, its $L^{2}-$norm is conserved, 
$$\|\psi_k(t)\|_{L^{2}}=\|\psi_{k0}\|_{L^{2}}, \   \forall t \in \mathbb R, k =1,2,...,N$$
and for all $\alpha-$ fractional admissible pairs  $(p,q),$ and $  (\psi_1,..., \psi_N) \in  \left( L_{loc}^{p}(\mathbb R, L^{q}(\mathbb R^d)) \right)^{N}.$

\item[(ii)] \label{rc}  Let  $d\ge 2$ and  $\frac{2d}{2d-1} < \alpha < 2.$  If $ (\psi_{01}, ...., \psi_{0N}) \in \left(L_{rad}^2(\mathbb R^d)\right)^N$ then \eqref{HF} has a unique global solution 
$$ (\psi_1,...,\psi_N)   \in  \left(C(\mathbb R, L_{rad}^{2}(\mathbb R^d))\cap L^{4\alpha/\gamma}_{loc}(\mathbb R, L^{4d/(2d-\gamma)} (\mathbb R^d))\right)^N.$$ 
In addition, its $L^{2}-$norm is conserved, 
$$\|\psi_k(t)\|_{L^{2}}=\|\psi_{k0}\|_{L^{2}}, \   \forall t \in \mathbb R, k =1,2,...,N$$
and for all $\alpha-$ fractional admissible pairs  $(p,q),$ and $  (\psi_1,..., \psi_N) \in  \left( L_{loc}^{p}(\mathbb R, L^{q}(\mathbb R^d)) \right)^{N}.$
\end{enumerate}
\end{proposition}

\begin{proof}
We first establish part (ii). 
By Duhamel's formula, we write \eqref{HF}
as 
$$\Phi(\psi_1,..., \psi_N):=\psi_{k}(t) = U(t)\psi_{0k} -i \int_0^tU(t-s) (H\psi_k)(s) ds 
+ i \int_0^t U(t-s) (F_k(\psi_k))(s) ds$$
where Hartree factor  $H= \sum_{l=1}^{N} \left(\frac{1}{|x|^{\gamma}} \ast |\psi_{l}|^2 \right)$ and  Fock term   $F(\psi_k)=\sum_{l=1}^{N} \psi_{l} \left(\frac{\kappa}{|x|^{\gamma}} \ast \{\overline{\psi_l} \psi_k\} \right)$.
Put $ s= \frac{\alpha}{2}$.  We introduce the space
\begin{eqnarray*}
Y(T)  & =  &\{ (\psi_1,..., \psi_N) \in  \left (C\left([0,T], L_{rad}^2(\mathbb R^d) \right)\right)^N :\|\psi_k \|_{L^{\infty}([0, T], L^2)}  \leq 2 \|\psi_{0k}\|_{L^2}, \\
&&  \|\psi_k\|_{L^{\frac{8s}{\gamma}} ([0,T], L^{\frac{4d}{2d-\gamma}}  ) } \lesssim \|\psi_{0k}\|_{L^2}\}
\end{eqnarray*}
and the distance 
$$d(\phi_1, \phi_2)=  \max\left \{ \|f_i - g_i \|_{L^{\frac{8s}{\gamma} }\left( [0, T], L^{\frac{4d}{(2d- \gamma)}}\right)}: i=1,..., N \right\},$$
where $\phi_1= (f_1,..., f_N)$ and $\phi_2= (g_1,...,g_N).$ Then $(Y, d)$ is a complete metric space. Now we show that $\Phi$ takes $Y(T)$ to $Y(T)$ for some $T>0.$ 
We put 
$$ \ q= \frac{8s}{\gamma}, \  r= \frac{4d}{2d- \gamma}.$$ 
Note that $(q,r)$ is $\alpha-$fractional admissible and 
$$ \frac{1}{q'}= \frac{4s- \gamma}{4s} + \frac{1}{q}, \  \frac{1}{r'}= \frac{\gamma}{2d} + \frac{1}{r}.$$
Let $(\bar{q}, \bar{r}) \in \{ (q,r), (\infty, 2) \}.$ By part (ii) of Proposition \ref{fst}  and  H\"older's inequality, we have 
\begin{eqnarray*}
I &:= & \|\Phi(\psi_1,...,\psi_N)\|_{L_{t,x}^{\bar{q}, \bar{r}}} \\
&  \lesssim &  \|\psi_{0k}\|_{L^2} + \|H\psi_k \|_{L_{t,x}^{q',r'}} + \|F\psi_k \|_{L_{t,x}^{q',r'}} \\
&  \lesssim &  \|\psi_{0k}\|_{L^2} + \sum_{l=1}^{N}\left\| \left( |\cdot|^{-\gamma}\ast |\psi_l|^2 \right) \psi_k \right \|_{L_{t,x}^{q',r'}} + \left\| \left( |\cdot|^{-\gamma}\ast  (\bar{\psi_l} \psi_k) \right) \psi_l \right \|_{L_{t,x}^{q',r'}}\\
& \lesssim &  \|\psi_{0k}\|_{L^2} + \sum_{l=1}^{N}  \| |\cdot|^{-\gamma} \ast |\psi_l|^2\|_{L_{t,x}^{\frac{4s}{4s-\gamma}, \frac{2d}{\gamma}}}  
\|\psi_k\|_{L^{q,r}_{t,x}} +  \| |\cdot|^{-\gamma} \ast (\bar{\psi_l} \psi_k)\|_{L_{t,x}^{\frac{4s}{4s-\gamma}, \frac{2d}{\gamma}}} 
\|\psi_l\|_{L^{q,r}_{t,x}}.
\end{eqnarray*}
Since $0<\gamma< \min \{\alpha, d \}$, by the Hardy-Littlewood-Sobolev lemma,  we  have 
\begin{eqnarray*}
\| |\cdot|^{-\gamma} \ast (\bar{\psi_{l}} \psi_k)\|_{L_{t,x}^{\frac{4s}{4s-\gamma}, \frac{2d}{\gamma}}}  & = &  \left\|  \| |\cdot|^{-\gamma}\ast (\bar{\psi_{l}} \psi_k)\|_{L_x^{\frac{2d}{\gamma}}} \right\|_{L_t^{\frac{4s}{4s- \gamma}}}\\
& \lesssim &   \left \| \||\bar{\psi_{l}} \psi_k\|_{L_x^{\frac{2d}{2d- \gamma}}} \right\|_{L_t^{\frac{4s}{4s- \gamma}}} \\
& \lesssim & \|\psi_{l}\|_{L_{t,x}^{\frac{8s}{4s- \gamma},r}} \|\psi_{k}\|_{L_{t,x}^{\frac{8s}{4s- \gamma},r}}\\
& \lesssim & T^{1- \frac{\gamma}{2s}} \|\psi_l\|_{L_{t,x}^{q,r}} \|\psi_k\|_{L_{t,x}^{q,r}}.
\end{eqnarray*}
Observe that in the last inequality we use the  inclusion relation for the finite measure space $L^p([0, T])$. 
Thus, we  have 
$$
\|\Phi(\psi_1,...,\psi_N)\|_{L_{t,x}^{\bar{q}, \bar{r}}}   \lesssim   \|\psi_{0k}\|_{L^2}+T^{1- \frac{\gamma}{2s}}  \sum_{l=1}^{N}  \|\psi_l\|^2_{L_{t,x}^{q,r}} \|\psi_k\|_{L_{t,x}^{q,r}}.$$
This shows that $\Phi$ maps $Y(T)$ to $Y(T).$  
Next, we show $\Phi$
is a  contraction. To this end, we  notice  the following identity: for fixed $j\in \{1,...,N\}$ and $K(x)=|x|^{-\gamma},$ we have 
\begin{multline}\label{di}
\sum_{i=1}^{N}(K\ast |u_i|^2) u_j - (K\ast |v_i|^2) v_j = \sum_{i=1}^{N} (K\ast |u_i|^{2})(u_j-v_j) + (K \ast (|u_i|^{2}- |v_i|^{2}))v_j
\end{multline}
and
\begin{multline}\label{di2}
\sum_{i=1}^{N}(K\ast  (\bar{u_i}u_j)) u_i - (K\ast ( \bar{v_i}v_j)) v_i = \sum_{i=1}^{N} (K\ast (\bar{u_i} u_j))(u_i-v_i) + (K \ast  (\bar{u_i}u_j-  (\bar{v_i} v_j) ))v_i.
\end{multline}
It follows that 
\begin{eqnarray}\label{mi}
 \|(K\ast (\bar{u_i} u_j))(u_i-v_i)\|_{L_{t,x}^{q',r'}} \lesssim  T^{1-\frac{\gamma}{2s}} \|u_i\|_{L_{t,x}^{q,r}} \|u_j\|_{L_{t,x}^{q,r}} \|u_i-v_i\|_{L_{t,x}^{q,r}}.
\end{eqnarray}
Put $\delta = \frac{8s}{4s-\gamma}.$  Notice that $\frac{1}{q'}= \frac{1}{2}+ \frac{1}{\delta}, \frac{1}{2}= \frac{1}{\delta} + \frac{1}{q},$ and thus by H\"older'sinequality, we obtain
\begin{eqnarray}\label{mi1}
 \|(K \ast (|u_i|^{2}- |v_i|^{2}))v_j\|_{L_{t,x}^{q',r'}} & \lesssim & \| K\ast \left( |u_i|^2-|v_i|^2\right)\|_{L_{t,x}^{2, \frac{2d}{\gamma}}} \|v_j\|_{L_{t,x}^{\delta, r}} \nonumber \\
 & \lesssim & ( \|K \ast (u_i (\bar{u_i}- \bar{v_i})) \|_{L_{t,x}^{2, \frac{2d}{\gamma}}} \nonumber
 \\
 && + \|K \ast \bar{v_i}(u_i-v_i)) \|_{L_{t,x}^{2, \frac{2d}{\gamma}}} ) \|v_j\|_{L_{t,x}^{\delta,r}} \nonumber\\
 & \lesssim & \left( \|u_i\|_{L_{t,x}^{\delta,r}} \|v_j\|_ {L_{t,x}^{\delta,r}} +\|v_i\|_ {L_{t,x}^{\delta,r}}\|v_j\|_ {L_{t,x}^{\delta,r}}\right) \|u_i-v_i\|_{L_{t,x}^{q,r}}\nonumber\\
 & \lesssim & T^{1-\frac{\gamma}{2s}}  \left( \|u_i\|_{L_{t,x}^{q,r}} \|v_j\|_ {L_{t,x}^{q,r}} +\|v_i\|_ {L_{t,x}^{q,r}}\|v_j\|_ {L_{t,x}^{q,r}}\right)  \|u_i-v_i\|_{L_{t,x}^{q,r}}.
\end{eqnarray}
Similarly,
\begin{eqnarray*}
\|(K \ast  (\bar{u_i}u_j-  (\bar{v_i} v_j) ))v_i\|_{L_{t,x}^{q',r'}} & \lesssim &  T^{1-\frac{\gamma}{2s}}   \|u_i\|_{L_{t,x}^{q,r}} \|v_i\|_ {L_{t,x}^{q,r}} \|u_j-v_j\|_{L_{t,x}^{q,r}} \\
&&+ T^{1-\frac{\gamma}{2s}}\|v_i\|_ {L_{t,x}^{q,r}}\|v_j\|_ {L_{t,x}^{q,r}} \|u_j-v_j\|_{L_{t,x}^{q,r}}.
\end{eqnarray*}
Let $u= (u_1,...,u_N)$ and $v=(v_1,...., v_N).$  Now in view of \eqref{di}, \eqref{mi}, and \eqref{mi1}, we have 
\begin{eqnarray*}
 d(\Phi(u), \Phi(v))& \lesssim & \sum_{i=1}^N  \|(K\ast |u_i|^{2})(u_j-v_j)\|_{L_{t,x}^{q',r'}} + \|(K \ast (|u_i|^{2}- |v_i|^{2}))v_j\|_{L_{t,x}^{q',r'}}\\
 && + \|(K\ast (\bar{u_i} u_j))(u_i-v_i)\|_{L_{t,x}^{q',r'}} + \|(K \ast  (\bar{u_i}u_j-  (\bar{v_i} v_j) ))v_i\|_{L_{t,x}^{q',r'}}.\\
& \lesssim &   T^{1-\frac{\gamma}{2s}} \sum_{i=1}^N  [\|u_i\|^2_{L_{t,x}^{q,r}}  +\|u_i\|_{L_{t,x}^{q,r}} \|v_j\|_ {L_{t,x}^{q,r}} +\|v_i\|^2_ {L_{t,x}^{q,r}} \|v_j\|^2_ {L_{t,x}^{q,r}}\\
&&  +  \|v_i\|_ {L_{t,x}^{q,r}}\|v_j\|_ {L_{t,x}^{q,r}} + \|v_i\|_ {L_{t,x}^{q,r}}\|v_j\|_ {L_{t,x}^{q,r}} ]d(u,v).
\end{eqnarray*}
Thus $\Phi$ is a contraction from $Y(T)$ to $Y(T)$  provided that $T$ is sufficiently small. Then there exists a unique $ (\psi_1,..., \psi_N) \in Y(T)$ solving \eqref{HF}. The global  existence of the solution \eqref{HF} follows from the conservation of the $L^2-$norm of $\psi_k.$ The last property of the proposition then follows from the Strichartz estimates applied with an arbitrary $\alpha-$fractional admissible pair on the left hand side and the same pairs as above on the right hand side. This completes the proof of  part (ii). 

 The proof of part (i) follows by setting  $\alpha=2$ and using Proposition \ref{fst} part (i).
\end{proof}

\begin{proposition}\label{miF2}  Let $\phi\circ h (\xi)= |\xi|^{\alpha}$ for $ \xi \in \mathbb R^d, \alpha>0$  and  $0 < \gamma < \min\{\alpha, d\}.$
\begin{enumerate}
\item[(i)] \label{wrc} Let  $d\in \mathbb N$ and  $\alpha =2.$  If $ (\psi_{01}, ...., \psi_{0N}) \in \left(L^2(\mathbb R^d)\right)^N$ then \eqref{RHF} has a unique global solution 
$$ (\psi_1,...,\psi_N)   \in  \left(C(\mathbb R, L^{2}(\mathbb R^d))\cap L^{4\alpha/\gamma}_{loc}(\mathbb R, L^{4d/(2d-\gamma)} (\mathbb R^d))\right)^N.$$ 
In addition, its $L^{2}-$norm is conserved, 
$$\|\psi_k(t)\|_{L^{2}}=\|\psi_{k0}\|_{L^{2}}, \   \forall t \in \mathbb R, k =1,2,...,N$$
and for all $\alpha-$ fractional admissible pairs  $(p,q),$ and $  (\psi_1,..., \psi_N) \in  \left( L_{loc}^{p}(\mathbb R, L^{q}(\mathbb R^d)) \right)^{N}.$

\item[(ii)] \label{rc}  Let  $d\ge 2$ and  $\frac{2d}{2d-1} < \alpha < 2.$  If $ (\psi_{01}, ...., \psi_{0N}) \in \left(L_{rad}^2(\mathbb R^d)\right)^N$ then \eqref{RHF} has a unique global solution 
$$ (\psi_1,...,\psi_N)   \in  \left(C(\mathbb R, L_{rad}^{2}(\mathbb R^d))\cap L^{4\alpha/\gamma}_{loc}(\mathbb R, L^{4d/(2d-\gamma)} (\mathbb R^d))\right)^N.$$ 
In addition, its $L^{2}-$norm is conserved, 
$$\|\psi_k(t)\|_{L^{2}}=\|\psi_{k0}\|_{L^{2}}, \   \forall t \in \mathbb R, k =1,2,...,N$$
and for all $\alpha-$ fractional admissible pairs  $(p,q),$ and   $(\psi_1,..., \psi_N) \in  \left( L_{loc}^{p}(\mathbb R, L^{q}(\mathbb R^d)) \right)^{N}.$
\end{enumerate}
\end{proposition}

\begin{proof}
Since the proof is similar to that of  Proposition \ref{miF}, we omit its details.
\end{proof}

Let $\Psi= (\psi_1,...,\psi_N): (\mathbb R \times \mathbb R^d)^N\to \mathbb  C$ be a global $L^2-$  solution given by  Proposition \ref{miF}.
Let $T_{+}$ denotes the maximal time of existence:
\[T_{+}(\Psi)=\sup \left\{ T>0: \Psi(t)\mid_{([0, T]\times \mathbb R^d)^N}\in \left( C([0, T], X)\right)^{N} \right\}.\]
Theorem \ref{gidi} tells us that  $T_{+}(\Psi)>0$ if initial data $(\psi_{01},..., \psi_{0N})\in \left(C([0, T], X\cap L^2(\mathbb R^d)) \right)^N.$

\begin{Lemma} \label{sb}  Assume that  $0< T_{+}< \infty.$ Then 
\begin{eqnarray*}
\lim_{t\to T_{+}} \sum_{k=1}^{N} \|\psi_{k}(t)\|_{X}= \infty.
\end{eqnarray*}
\end{Lemma}

\begin{proof} We proceed by contradiction and assume that there exist $M>0$ and $\{t_n\}_{n=1}^{\infty}$ such that
\begin{eqnarray*}
t_{n}\to T_{+}  \ \text{as} \ n\to \infty \ \ \text{and} \ \  \sum_{k=1}^{N}\|\psi_k(t_n)\|_{X}\leq M.
\end{eqnarray*}  
Recall that  the life span of the local solution in  Theorem \ref{gidi} depends on the norm of the initial data.  Therefore, there is $T=T(M)>0$ such that  for each $n\in \mathbb N,$ the solution   $\Psi(t)= (\psi_1(t),..., \psi_N(t))$  of \eqref{HF} can be established on the time interval $[t_n, t_n+ T(M)].$ By uniqueness, $\psi_k(t)$ coincides  with standard global $L^2-$solution on this interval, which implies 
\begin{eqnarray*}
\psi_k(t)\mid_{[0, T_{+}+\epsilon] \times \mathbb R^d} \in C([0, T_{+} + \epsilon], X)
\end{eqnarray*}
for some $\epsilon\in  (0, T(M))$ and for $k=1,2..., N$  but this is a contradiction.
\end{proof}
Now we shall see that the solution constructed before is global in time.  In fact, in view of Proposition \ref{miF},  to prove Theorem \ref{dgt}, it suffices to prove  that the modulation space norm of $\psi_k$, that is, $\|\psi_{k}\|_{M^{p,q}}$ cannot become unbounded in finite time for all $k=1,...,N.$ To this end, let $T_0>0$ and $\psi_k:[0, T_0]\times \mathbb R^d \to \mathbb C$ be a local solution to  \eqref{HF} such that $$\psi_k(t)\in C\left([0, T],X \cap L^2(\mathbb R^d)\right)$$ for any $T\in (0, T_0)$ and for $k=1,...,N.$
\begin{Lemma} \label{ms}Assume that $0<\gamma < \min\{\alpha, d/2\}.$ Then 
\begin{eqnarray*}
\sup_{t\in [0, T_{0})} \sum_{k=1}^{N} \|\psi_k(t)\|_{X} < \infty.
\end{eqnarray*}
\end{Lemma}
\begin{proof}
 There exists $C=C(d,\gamma)$ such that the Fourier transform of $K(x)= \kappa |x|^{-\gamma}$ is
\begin{eqnarray*}
\widehat{K}(\xi)= \frac{\kappa C}{|\xi|^{d-\gamma}}.
\end{eqnarray*}
We can decompose the  Fourier transform of Hartree potential into Lebesgue spaces: indeed,  we have 
\begin{eqnarray}\label{dc}
\widehat{K}=k_1+k_2 \in L^{p}(\mathbb R^{d})+ L^{q}(\mathbb R^{d}),
\end{eqnarray}
where  $k_{1}:= \chi_{\{|\xi|\leq 1\}}\widehat{K} \in L^{p}(\mathbb R^{d})$ for all $p\in [1, \frac{d}{d-\gamma})$ and $k_{2}:= \chi_{\{|\xi|>1\}} \widehat{K} \in L^{q}(\mathbb R^{d})$ for all $q\in (\frac{d}{d-\gamma}, \infty].$

In view of \eqref{dc} and to use the Hausdorff-Young inequality we let $1< \frac{d}{d-\gamma} <q \leq 2,$ 
and we obtain
\begin{eqnarray*}
I_k & := & \|\psi_k(t)\|_{X} \\
& \lesssim &  C_{T} \left( \|\psi_{0k}\|_{X} + \sum_{l=1}^{N} \int_{0}^{t} \|(K\ast |\psi_{l}(\tau)|^{2}) \psi_{k}(\tau)\|_{X}+\|(K\ast (\bar{\psi_l}\psi_k)) \psi_{l}(\tau)\|_{X}  d\tau \right)\nonumber \\
& \lesssim &  C_{T}  \|\psi_{0k}\|_{X} + C_T \sum_{l=1}^{N} \int_{0}^{t}  \|K\ast |\psi_l(\tau)|^2\|_{\mathcal{F}L^{1}} \|\psi_k(\tau)\|_{X} + \|K\ast (\bar{\psi_l}\psi_k)\|_{\mathcal{F}L^{1}} \|\psi_l(\tau)\|_{X} d\tau \\
& \lesssim & C_{T}   \|\psi_{0k}\|_{X} + C_{T} \sum_{l=1}^{N} \int_{0}^{t} \left( \|k_{1}\|_{L^{1}} \|\psi_{l}(\tau)\|_{L^{2}}^{2}+ \|k_{2}\|_{L^{q}} \|\widehat{|\psi_{l}(\tau)|^{2}}\|_{L^{q'}}
\right) \|\psi_k(\tau)\|_{X} d\tau \nonumber\\
&& +   C_{T} \sum_{l=1}^{N} \int_{0}^{t} \left( \|k_{1}\|_{L^{1}} \|\psi_{l}(\tau)\psi_k(\tau)\|_{L^{1}}+ \|k_{2}\|_{L^{q}} \|\widehat{\bar{\psi_{l}}(\tau) \psi_k(\tau)}\|_{L^{q'}}
\right) \|\psi_l(\tau)\|_{X} d\tau \nonumber\\
& \lesssim & C_{T}\|\psi_{0k}\|_{X} + C_{T} \sum_{l=1}^{N}\int_{0}^{t} \left(  \|k_{1}\|_{L^{1}} \|\psi_{0l}\|_{L^{2}}^{2}+ \|k_{2}\|_{L^{q}} \||\psi_{l}(\tau) |^{2}\|_{L^{q}}\right)  \|\psi_k(\tau)\|_{X} d\tau\nonumber\\
&&  + C_{T} \sum_{l=1}^{N}\int_{0}^{t} \left(  \|k_{1}\|_{L^{1}} \|\psi_{0l}\|_{L^{2}} \|\psi_{0k}\|_{L^{2}}  + \|k_{2}\|_{L^{q}} \|{\bar{\psi_{l}}(\tau) \psi_k(\tau)}\|_{L^{q}}\right)  \|\psi_l(\tau)\|_{X} d\tau\nonumber\\
& \lesssim & C_{T}\|\psi_{0k}\|_{X}+ C_{T}(N) \int_{0}^{t}\|\psi_k(\tau)\|_{X} d\tau + C_{T} \sum_{l=1}^N \int_{0}^{t} \|\psi_l(\tau)\|_{L^{2q}}^{2} \|\psi_k(\tau)\|_{X} d\tau\\
&& + C_{T}  \sum_{l=1}^N\int_{0}^{t}\|\psi_l(\tau)\|_{X} d\tau + C_{T} \sum_{l=1}^N \int_{0}^{t} \|\psi_l(\tau)\|_{L^{2q}} \|\psi_k(\tau)\|_{L^{2q}}\|\psi_l(\tau)\|_{X} d\tau
 \end{eqnarray*}
 where we have used  Proposition \ref{gap},  H\"older's inequality, and  the conservation of the $L^{2}-$norm of $\psi_k$ ($k=1,..., N$) and $C_T$ is defined as in the proof of Theorem \ref{gidi}.
We note that the requirement on $q$ can be fulfilled if and only if $0<\gamma <d/2.$ To apply Proposition \ref{mi}, we let $\beta>1$ and $(2\beta, 2 q)$ is  $\alpha-$fractional admissible, that is, $\frac{\alpha}{2\beta}= d \left(\frac{1}{2}- \frac{1}{2q} \right)$ such that $\frac{1}{\beta}= \frac{d}{\alpha} \left( 1 - \frac{1}{q} \right)<1.$ This is possible provided $\frac{q-1}{q} < \frac{\alpha}{d}:$ this condition is compatible with the requirement $q> \frac{d}{d-\gamma}$ if and only if $\gamma < \alpha.$
 Using  H\"older's inequality for the last integral, we obtain
\begin{eqnarray*}
I_k &  \lesssim  & C_{T}\|\psi_{0k}\|_{X} + C_{T}(N)\int_{0}^{t} \|\psi_k(\tau)\|_{X} d\tau  + C_{T} \sum_{l=1}^N\|\psi_{l}\|_{L^{2\beta}([0, T], L^{2q})}^{2}\|\psi_k\|_{L^{\beta'}([0, T], X)}\\
&&  + C_{T}  \sum_{l=1}^N\int_{0}^{t}\|\psi_l(\tau)\|_{X} d\tau  +   C_{T} \sum_{l=1}^N\|\psi_{l}\|_{L^{2\beta}([0, T], L^{2q})} \|\psi_{k}\|_{L^{2\beta}([0, T], L^{2q})}\|\psi_l\|_{L^{\beta'}([0, T], X)}
\end{eqnarray*}
where $\beta'$ is the H\"older conjugate exponent of $\beta.$
Let
\begin{eqnarray*}
h(t)=\sup_{s\in [0, t]} \sum_{k=1}^{N} \|\psi_k(s)\|_{X}.
\end{eqnarray*}
For a given $T>0,$ $h$ satisfies an estimate of the form,
$$h(t)\lesssim C_{T} \sum_{k=1}^N \|\psi_{0k}\|_{X}+ C_{T}(N)\int_{0}^{t} h(\tau) d\tau + C_{T} C_0(T, N) \left( \int_{0}^{t}h(\tau)^{\beta'} d\tau \right)^{\frac{1}{\beta'}},$$
provided that $0 \leq t \leq T,$ and where we have used the fact that $\beta'$ is finite.
Using  H\"older's inequality we infer that,
$$h(t)\lesssim  C_{T} \sum_{k=1}^N\|\psi_{0k}\|_{X} + C_{1}(T, N) \left(\int_{0}^{t} h(\tau)^{\beta'}d \tau \right)^{\frac{1}{\beta'}}.$$
Raising the above estimate to the power $\beta'$, we find that
$$h(t)^{\beta'} \lesssim  C_{2}(T, N) \left( 1+\int_{0}^{t} h(\tau)^{\beta'} d\tau\right).$$ 
In view of  Gronwall inequality, one may conclude that  $h\in L^{\infty}([0, T]).$ Since $T>0$ is arbitrary, $h\in L^{\infty}_{loc}(\mathbb R).$  This completes  the proof.
\end{proof}

We can now prove Theorem~\ref{dgt}.

\begin{proof}[\textbf{Proof of Theorem \ref{dgt}}]  Taking Theorem \ref{gidi} into account and combining   Lemmas \ref{ms} and  \ref{sb}, the  proof  of Theorem \ref{dgt} part (i) follows.  Similarly, we can produced the proof of Theorem \ref{dgt} part (ii), we shall omit the details.
\end{proof}

\section{Well-posedness for  Hartree-Fock equations with harmonic potential}\label{whfhp}
In this final section we consider the Hatree-Fock  and reduced Hartree-Fock equations with a harmonic potential as given by~\eqref{HFHP} and \eqref{RHFHP}.

\subsection{Schr\"odinger propagator associated to harmonic oscillator}\label{whm} We start by recalling the spectral decomposition of $H=-\Delta + |x|^2$ by  the Hermite expansion. Let $\Phi_{\alpha}(x),  \  \alpha \in \mathbb N^d$ be the normalized Hermite functions which are products of one dimensional Hermite functions. More precisely, 
$ \Phi_\alpha(x) = \Pi_{j=1}^d  h_{\alpha_j}(x_j) $ 
where 
$$ h_k(x) = (\sqrt{\pi}2^k k!)^{-1/2} (-1)^k e^{\frac{1}{2}x^2}  \frac{d^k}{dx^k} e^{-x^2}.$$ The Hermite functions $ \Phi_\alpha $ are eigenfunctions of $H$ with eigenvalues  $(2|\alpha| + d)$  where $|\alpha |= \alpha_{1}+ ...+ \alpha_d.$ Moreover, they form an orthonormal basis for $ L^2(\R^d).$ The spectral decomposition of $ H $ is then written as
$$ H = \sum_{k=0}^\infty (2k+d) P_k\quad \textrm{with}\quad  P_kf(x) = \sum_{|\alpha|=k} \langle f,\Phi_\alpha\rangle \Phi_\alpha$$ 
where $\langle\cdot, \cdot \rangle $ is the inner product in $L^2(\mathbb{R}^d)$. Given a function $m$ defined and bounded  on the set of all natural numbers we can use the spectral theorem to define $m(H).$ The action of $m(H)$ on a function $f$ is given by
\begin{eqnarray*}\label{dhm}
m(H)f= \sum_{\alpha \in \mathbb N^d} m(2|\alpha| +d) \langle f, \Phi_{\alpha} \rangle \Phi_{\alpha} = \sum_{k=0}^\infty m(2k+d)P_kf.
\end{eqnarray*}
This operator  $m(H)$ is bounded on $L^{2}(\mathbb R^d).$ This follows immediately from the Plancherel theorem for the Hermite expansions as  $m$ is bounded.  On the other hand, the mere boundedness of $m$ is not sufficient  to imply  the $L^{p}$ boundedness of $m(H)$ for $p\neq 2$ (see   \cite{thangavelu1993lectures}). We define \textbf{Schr\"odinger propagator associated to harmonic oscillator}
$$m(H)=e^{it(-\Delta +|x|^2)}f= \sum_{k=0}^\infty e^{it (2k+d)}P_kf$$ with $m(n)=e^{itn}$ for $n\in \mathbb N, t\in  \mathbb R.$ The next result proves that $e^{it (-\Delta +|x|^2)}$ is uniformly  bounded on $M^{p,p}(\R^d).$ More specifically, we have.

\begin{theorem}(\cite[Theorem 5]{drt}, cf. \cite{cordero2008metaplectic})\label{mso}
The Schr\"odinger  propagator  associated to harmonic oscillator $e ^{it (-\Delta +|x|^2)}$ is bounded on $M^{p,p}(\mathbb R^d)$ for each $t\in \mathbb R$, and all  $1\leq p < \infty.$ Moreover, we have $$\|e^{it (-\Delta +|x|^2)}f\|_{M^{p,p}}=\|f\|_{M^{p,p}}.$$
\end{theorem}
\subsection{Proof of Theorem \ref{mtg}}
In this section we give a proof of Theorem~\ref{mtg}. But first, we state the following definition and some preliminary results.

\begin{Definition} A pair $(q,r)$ is admissible if  $2\leq r< \frac{2d}{d-2}$ with $2\leq r \leq \infty$ if $d=1$, and  $2\leq r < \infty$ if $d=2$, whenever 
$$\frac{2}{q} =  d \left( \frac{1}{2} - \frac{1}{r} \right).$$ 
\end{Definition}

\begin{proposition}(\cite[Proposition 2.2]{carles2011nonlinear}) \label{seh}    Let $\phi \in L^2(\mathbb R^d)$ and  $$DF(t,x) :=  U(t)\phi(x) +  \int_0^t U(t-\tau )F(\tau,x) d\tau.$$  Then for any time slab $I$ and  admissible pairs $(p_i, q_i)$, $i=1,2, $ 
there exists  a constant $C=C(|I|, q_1)$ such that for all intervals $I \ni 0, $
$$ \|D(F)\|_{L^{p_1, q_1}_{t,x}}  \leq  C \|\phi \|_{L^2}+   C  \|F\|_{L^{p_2', q_2'}_{t,x}}, \ \forall F \in L^{p_2'} (I, L^{q_2'})$$  where $p_i'$ and $ q_i'$ are H\"older conjugates of $p_i$ and $q_i$
respectively.
\end{proposition}

\begin{proposition}  \label{miD} Let 
 $ 0<\gamma < \text{min} \{2, d\}, d\in \mathbb N$. Assume that $  (\psi_{01},..., \psi_{0N}) \in  \left(L^{2}(\mathbb R^{d}) \right)^N.$ Then
\begin{enumerate}
\item[(i)]  There exists  a unique global solution  of \eqref{HFHP} such that
$$ (\psi_1,...,\psi_N)   \in  \left(C([0, \infty), L^{2}(\mathbb R^d))\cap L^{4\alpha/\gamma}_{loc}([0, \infty), L^{4d/(2d-\gamma)} (\mathbb R^d))\right)^N.$$ 
In addition, its $L^{2}-$norm is conserved, 
$$\|\psi_k(t)\|_{L^{2}}=\|\psi_{k0}\|_{L^{2}}, \   \forall t \in \mathbb R, k =1,2,...,N$$
and for all  admissible pairs  $(p,q),$ and $  (\psi_1,..., \psi_N) \in  \left( L_{loc}^{p}(\mathbb R, L^{q}(\mathbb R^d)) \right)^{N}.$
\item[(ii)]  There exists  a unique global solution  of \eqref{RHFHP} such that 
$$ (\psi_1,...,\psi_N)   \in  \left(C([0, \infty), L^{2}(\mathbb R^d))\cap L^{4\alpha/\gamma}_{loc}([0, \infty), L^{4d/(2d-\gamma)} (\mathbb R^d))\right)^N.$$ 
In addition, its $L^{2}-$norm is conserved, 
$$\|\psi_k(t)\|_{L^{2}}=\|\psi_{k0}\|_{L^{2}}, \   \forall t \in \mathbb R, k =1,2,...,N$$
and for all  admissible pairs  $(p,q),$ and $  (\psi_1,..., \psi_N) \in  \left( L_{loc}^{p}(\mathbb R, L^{q}(\mathbb R^d)) \right)^{N}.$
\end{enumerate}
   \end{proposition}

\begin{proof}  The proof follows from Proposition~\ref{seh}  and using ideas similar to the proof of Proposition~\ref{miF}.
\end{proof}
We can now establish  local well-posedness results for~\eqref{HFHP} and~\eqref{RHFHP}. 

\begin{theorem}[Local well-posedness]\label{gidih}  Let $1\leq p \leq \frac{2d}{d+\gamma}$ and $0<\gamma<d.$  Assume that   $\left(\psi_{01},..., \psi_{0N} \right) \in \left(M^{p,p}(\mathbb R^d)\right)^N$. 
  Then
  \begin{enumerate}
  \item[(i)] \label{gidi1}There exists $T>0$ depending only on $\|\psi_{01}\|_{M^{p,p}},...,\|\psi_{0N}\|_{M^{p,p}},$ $d$ and $\gamma$ such that \eqref{HFHP} has a unique local solution
$$(\psi_1,..., \psi_{N})\in \left( C([0, T], M^{p,p}(\mathbb R^d)) \right)^N.$$ 
\item[(ii)] \label{gidi2} There exists $T>0$ depending only on $\|\psi_{01}\|_{M^{p,p}},...,\|\psi_{0N}\|_{M^{p,p}},$ $d$ and $\gamma$ such that \eqref{RHFHP} has a unique local solution
$$(\psi_1,..., \psi_{N})\in \left( C([0, T], M^{p,p}(\mathbb R^d)) \right)^N.$$
\end{enumerate}   
 \end{theorem}

 \begin{proof}
 The results are established by applying a standard contraction mapping argument and using Theorem \ref{mso} and Proposition \ref{t1}.
 \end{proof}

\begin{proof}[\textbf{Sketch proof of Theorem \ref{mtg}}] The proof is similar to that of Theorem \ref{dgt} using Proposition \ref{seh} and Theorem \ref{gidih}.
\end{proof}

{\textbf{Acknowledgment}:}    D.G. B is very grateful to Professor  Kasso Okoudjou for hosting and  arranging research facilities at the University of Maryland.  D.G. B is  thankful to  SERB Indo-US Postdoctoral Fellowship (2017/142-Divyang G Bhimani) for the financial support. D.G.B is also thankful to DST-INSPIRE and TIFR CAM for the  academic leave.  K. A. O.\ was partially supported by a grant from the Simons Foundation $\# 319197$,  the U. S.\ Army Research Office  grant  W911NF1610008,  the National Science Foundation grant DMS 1814253, and an MLK  visiting professorship.

\bibliographystyle{amsplain}
\bibliography{dg}
\end{document}